\documentclass[
11pt,twoside, a4paper,
]{amsart}

\title[Extended equivariant  Picard complexes]
{ Extended equivariant Picard complexes\\
and homogeneous spaces
}
\author{Mikhail Borovoi}
\address{Borovoi:
Raymond and Beverly Sackler School of Mathematical Sciences,
Tel Aviv University, 69978 Tel Aviv, Israel}
\email{borovoi@post.tau.ac.il}
\thanks{Joost van Hamel passed away in January 2008}
\thanks{Borovoi was partially supported by the
Hermann Minkowski Center for Geometry
and by ISF grant 807/07}

\author{Joost van Hamel}
\address{van Hamel: K.U. Leuven, Departement Wiskunde,
  Celestijnenlaan 200B, B-3001 Leuven (Heverlee), Belgium}
\thanks{This scientific collaboration was
supported by Project  G.0318.06 of the Flemish Science Fund FWO}


\usepackage[arrow,matrix]{xy}

\usepackage{mathptmx}
\usepackage{amscd}
\usepackage{amssymb}
\usepackage{eucal}
\usepackage{amsxtra}

\usepackage{verbatim}
\usepackage{url}

\DeclareSymbolFont{rsfs}{U}{rsfs}{m}{n}
\DeclareSymbolFontAlphabet{\mathcal}{rsfs}

\newcommand{\scH}{{\mathcal{H}{\mkern-3.5mu}}} 

\DeclareFontEncoding{OT2}{}{} 
\DeclareTextFontCommand{\textcyr}{\fontencoding{OT2}
    \fontfamily{wncyr}\fontseries{m}\fontshape{n}\selectfont}

\usepackage{verbatim}

\theoremstyle{plain}

\newtheorem{theorem}{Theorem}[section]
\newtheorem{proposition}[theorem]{Proposition}
\newtheorem{lemma}[theorem]{Lemma}
\newtheorem{corollary}[theorem]{Corollary}

\newtheorem{question}[theorem]{Question}

\newtheorem{conditional-result}[theorem]{Conditional Result}

\newtheorem{theorem?}{Theorem(?)}[section]
\newtheorem{proposition?}[theorem]{Proposition(?)}
\newtheorem{lemma?}[theorem]{Lemma(?)}
\newtheorem{corollary?}[theorem]{Corollary(?)}

\newtheorem{introtheorem}{Theorem}

\newtheorem{intromaintheorem}[introtheorem]{Main Theorem}

\newtheorem*{theorem*}{Theorem}
\newtheorem*{proposition*}{Proposition}
\newtheorem*{lemma*}{Lemma}
\newtheorem*{corollary*}{Corollary}
\newtheorem*{question*}{Question}
\newtheorem*{conjecture*}{Conjecture}
\newtheorem*{claim*}{Claim}

\newtheorem*{introtheorem*}{Theorem}
\newtheorem*{introproposition*}{Proposition}
\newtheorem*{introlemma*}{Lemma}
\newtheorem*{introcorollary*}{Corollary}

\theoremstyle{definition}

\newtheorem{definition}[theorem]{Definition}
\newtheorem{example}[theorem]{Example}
\newtheorem{construction}[theorem]{Construction}
\newtheorem{notation}[theorem]{Notation}
\newtheorem*{definition*}{Definition}
\newtheorem*{example*}{Example}

\newtheorem{subsec}[theorem]{}

\theoremstyle{remark}
\newtheorem{remark}[theorem]{Remark}

\newtheorem{remarks}[theorem]{Remarks}

\newcounter{thlistitem}
\renewcommand{\thethlistitem}{\roman{thlistitem}}
\newenvironment{theoremlist}{
\begin{list}{\makebox[1.5em]{\hfill\textup{(\thethlistitem)}}}%
{\usecounter{thlistitem}
\setlength{\leftmargin}{0em}
\setlength{\rightmargin}{0pt}
\setlength{\labelwidth}{-1em}
\setlength{\labelsep}{.5em}
}}{\end{list}}

\newcounter{assumlistitem}
\renewcommand{\theassumlistitem}{\roman{assumlistitem}}
  {\begin{list}
      {\hfill\textup{(\theassumlistitem)}}
      {\usecounter{assumlist}
       \setlength{\leftmargin}{0pt}
        \setlength{\rightmargin}{0pt}
        \setlength{\labelwidth}{-1.5em}}}
  {\end{list}}

\newtheorem*{remark*}{Remark}
\newtheorem*{Remarks*}{Remarks}
\newenvironment{remarks*}{\begin{Remarks*}\nopagebreak[4]
\rule{1em}{0ex}\par 
\begin{theoremlist}}%
{\end{theoremlist}\end{Remarks*}}

\newcommand{\mathb}[1]{\mathbf{#1}}

\newcommand{\capbar}{\overline}

\newcommand{\sH}{\mathcal{H}}
\newcommand{\sK}{\mathcal{K}}
\newcommand{\sO}{\mathcal{O}}

\newcommand{\Acat}{\mathcal{A}}

\newcommand{\fcolon}{\colon} 
\newcommand{\iso}{\cong}
\newcommand{\isoto}{\overset{\sim}{\to}}

\newcommand{\into}{\hookrightarrow}
\newcommand{\onto}{\twoheadrightarrow}
\newcommand{\labelTo}[1]{\xrightarrow{\makebox[4em]{\scriptsize ${#1}$}}}
\newcommand{\labelto}[1]{\xrightarrow{\makebox[1.5em]{\scriptsize ${#1}$}}}

\newcommand{\id}{\operatorname{id}}

\newcommand{\Zz}{{\mathb{Z}}}
\newcommand{\Qq}{{\mathb{Q}}}

\newcommand{\Cc}{{\mathb{C}}}

\DeclareMathOperator{\Hom}{Hom}
\DeclareMathOperator{\Ext}{Ext}
\DeclareMathOperator{\sHom}{{\scH\mathit{om}}}

\DeclareMathOperator*{\tensor}{\otimes}

\newcommand{\tors}{{}_{\textup{tors}}}

\newcommand{\hvar}{{-}} 

\newcommand{\Gm}{\mathb{G}_\mathrm{m}}

\newcommand{\uu}{^\mathrm{u}}

\def\red{^\mathrm{red}}
\def\tor{^{\mathrm{tor}}}
\newcommand{\scon}{^{\mathrm{sc}}}
\def\sss{^{\mathrm{ss}}}

\def\mult{^{\mathrm{mult}}}
\def\ssu{^{\mathrm{ssu}}}
\newcommand{\mm}{^{\mathrm{m}}}
\def\tors{_{\mathrm{tors}}}

\DeclareMathOperator{\Gal}{Gal}

\DeclareMathOperator{\Div}{Div}

\DeclareMathOperator{\Pic}{Pic}

\DeclareMathOperator{\divisor}{div}

\DeclareMathOperator{\Br}{Br}
\newcommand{\Bra}{\Br_\mathrm{a}}
\DeclareMathOperator{\Chi}{\mathbb{X}}

\DeclareMathOperator{\UPic}{UPic}
\DeclareMathOperator{\KDiv}{KDiv}
\newcommand{\OxDiv}{{\mathcal{O}_x}\textup{Div}}

\newcommand{\kbar}{{\bar{k}}}

\newcommand{\Xbar}{{\capbar{X}}}
\newcommand{\Ybar}{{\capbar{Y}}}
\newcommand{\Zbar}{{\capbar{Z}}}

\newcommand{\Gbar}{{\capbar{G}}}
\newcommand{\Hbar}{{\capbar{H}}}

\newcommand{\Tbar}{{\capbar{T}}}

\newcommand{\Fbar}{{\capbar{F}}}
\newcommand{\Pbar}{{\capbar{P}}}
\newcommand{\pibar}{{\overline{\pi}}}

\newcommand{\xb}{{\overline{x}}}

\def\Gmbar{{\capbar{\mathbf{G}}_{\textup{m}}}}

\newcommand{\updot}{^{\scriptscriptstyle{\bullet}}}
\newcommand{\upprimedot}{^{\prime\scriptscriptstyle{\bullet}}}

\renewcommand{\ggg}{{\mathfrak{g}}}
\def\Mbul{M\updot}
\def\Hombul{\Hom\updot}

\def\Ybul{Y\updot}

\newcommand{\XX}{{\Chi}}

\newcommand{\Zalg}{Z_{\textup{alg}}}
\newcommand{\Calg}{C_{\textup{alg}}}
\newcommand{\Balg}{B_{\textup{alg}}}
\newcommand{\Halg}{H_{\textup{alg}}}

\newcommand{\im}{{\textup{im\;}}}

\newcommand{\OX}{\sO(X)^\times}
\newcommand{\OXbar}{\sO(\Xbar)^\times}
\newcommand{\KX}{\sK(X)^\times}
\newcommand{\KXbar}{\sK(\Xbar)^\times}

\newcommand{\ZGOXbar}{\Zalg^1(\Gbar,\OXbar)}

\newcommand{\sL}{{\mathcal{L}}}
\newcommand{\sR}{{\mathcal{R}}}
\renewcommand{\div}{{\text{\rm div}}}

\def\xbar{\xb}
\def\alg{_{\text{\rm alg}}}
\def\Gmbar{\mathb{G}_{{\rm m},\kbar}}

\def\diag{{\rm diag}}
\def\der{{}^{\rm der}}
\def\Nbar{{\overline{N}}}
\newcommand{\res}{{\rm res}}
\newcommand{\Mbar}{{\overline{M}}}
\newcommand{\hh}{{\mathfrak{h}}}
\newcommand{\abar}{\overline{a}}

\begin{document}

\maketitle

\begin{abstract}
Let $k$ be a field of characteristic 0 and let $\kbar$
be a fixed algebraic closure of $k$.
Let $X$ be a smooth geometrically integral $k$-variety;
we set $\Xbar=X\times_k \kbar$ and denote by $\sK(\Xbar)$ the field of rational functions on $\Xbar$.
In \cite{BvH07} we defined  the \emph{extended Picard complex of $X$}
as the complex of $\Gal(\kbar/k)$-modules
$$
\UPic(\Xbar):=\left( \sK(\Xbar)^\times/\kbar^\times\labelto{\divisor}\Div(\Xbar)\right),
$$
where $\sK(\Xbar)^\times/\kbar^\times$ is in degree 0 and $\Div(\Xbar)$ is in degree 1.
We computed the isomorphism class of $\UPic(\Gbar)$ in the derived category of Galois modules
for a connected linear $k$-group $G$.

Here we compute the isomorphism class of $\UPic(\Xbar)$ in the derived category of Galois modules
when $X$ is a homogeneous space
of a connected linear $k$-group $G$ with $\Pic(\Gbar)=0$.
Let $\xb\in X(\kbar)$ and let $\Hbar$ denote the stabilizer of $\xb$ in $\Gbar$.
It is well known that the character group $\XX(\Hbar)$ of $\Hbar$ has a natural structure
of a Galois module.
We prove that
$$
\UPic(\Xbar)\cong\left( \XX(\Gbar)\labelto{{\rm res}}\XX(\Hbar)\right)
$$
in the derived category,
where res is the restriction homomorphism.
The proof is based on the notion of the extended equivariant Picard complex of a $G$-variety.
\end{abstract}

\section*{Introduction}
In this paper $k$ is always a field of characteristic 0, and $\kbar$ is a fixed algebraic closure of $k$.
A $k$-variety $X$ is always a geometrically integral $k$-variety, we set $\Xbar=X\times_k \kbar$.
We write $\sO(X)$ for the ring of regular functions on $X$,
and $\sK(X)$ for the field of rational functions on $X$.
By $\Div(X)$ we denote the group of Cartier divisors on $X$,
and by $\Pic(X)$ the Picard group of $X$.
By a $k$-group we mean a \emph{linear} algebraic group over $k$.

The \emph{extended Picard complex} $\UPic(\Xbar)$ of a smooth geometrically integral
variety $X$ over a field $k$
was introduced in the research announcement \cite{BvH06} and the paper \cite{BvH07}.
It is an object of the derived category of discrete Galois modules
(i.e. a complex of Galois modules) with cohomology
\begin{align*}
\sH^0(\UPic(\Xbar))  & = U(\Xbar) := \sO(\Xbar)^\times/\kbar^\times, \\
\sH^1(\UPic(\Xbar)) & = \Pic(\Xbar), \\
\sH^i(\UPic(\Xbar)) & = 0 \text{\ if $i \neq 0,1$.}
\end{align*}
This object $\UPic(\Xbar)$ is given by the complex of Galois modules
$$
\left[\sK(\Xbar)^\times/\kbar^\times\labelto{\divisor}\Div(\Xbar)\right\rangle,
$$
that is, the complex of length 2 with $\sK(\Xbar)^\times/\kbar^\times$ in degree 0
and $\Div(\Xbar)$ in degree 1; the differential $\divisor$ is the divisor map,
associating to the class $[f]$ of a rational function $f\in\sK(\Xbar)^\times$ the divisor $\divisor(f)\in\Div(\Xbar)$.
This complex plays an important role in understanding
arithmetic invariants such as the so-called \emph{algebraic part} $\Bra(X)$ of
the Brauer group $\Br(X)$ and the elementary obstruction of
Colliot-Th\'el\`ene and Sansuc \cite{CT-Sansuc} to the existence of
$k$-points in $X$.

In~\cite[Cor.~3]{BvH07} it was shown that when $X$ is
 a $k$-torsor under a  connected $k$-group $G$,
there is a canonical isomorphism (in the derived category of Galois modules)
\begin{equation}\label{eq:thm-old}
    \UPic(\Xbar) \iso [\Chi(\Tbar) \to \Chi(\Tbar\scon) \rangle,
  \end{equation}
where $[\Chi(\Tbar) \to \Chi(\Tbar\scon) \rangle$ is a certain complex constructed from $G$
using  a group-theoretic construction (see
Theorem~\ref{thm:old-main-result} below  for the precise statement).
In this way, a straightforward explanation was given why
this complex $[\Chi(\Tbar) \to \Chi(\Tbar\scon) \rangle$ played such an important role in the
study of Picard groups and Brauer groups and also of the {Brauer-Manin
obstruction} for torsors under linear algebraic groups.

The main result of the present paper is
the following theorem  announced in ~\cite[Theorem~3.1]{BvH06} concerning not necessarily
principal homogeneous spaces under a connected  $k$-group over a
field $k$ of characteristic 0.

\begin{intromaintheorem}[Theorem~\ref{th:upic-hom-space}]
\label{introth:upic-hom-space}
    Let $X$ be a right homogeneous space under a connected (linear) $k$-group $G$ with
   $\Pic(\Gbar) = 0$. Let $\Hbar\subset\Gbar$ be the  stabilizer of a geometric point  $\xb \in X(\kbar)$ in $\Gbar$
(we do not assume that $\Hbar$ is connected).
  Then we have a canonical isomorphism in the derived category of complexes of Galois modules
  \begin{equation*}
    \UPic(\Xbar) \iso [\Chi(\Gbar) \to \Chi(\Hbar) \rangle,
  \end{equation*}
 which is functorial in $G$ and $X$.
\end{intromaintheorem}

Here $\XX(\Gbar)$ and $\XX(\Hbar)$ are the character groups of $\Gbar$ and $\Hbar$, resp.
It is well-known that $\Chi(\Hbar)$ has a natural structure of a Galois module.
In the present paper we actually get this Galois structure
in a natural way by using V.L.~Popov's identification $\Chi(\Hbar) = \Pic_G(\Xbar)$,
where the right hand side is the \emph{equivariant Picard group} of $\Xbar$.
By $[\Chi(\Gbar) \to \Chi(\Hbar) \rangle$ we denote the complex of Galois modules
with $\XX(\Gbar)$ in degree 0 and with $\XX(\Hbar)$ in degree 1;
the differential $\Chi(\Gbar) \to \Chi(\Hbar)$ is the restriction of characters from $\Gbar$ to $\Hbar$.

 From Main Theorem \ref{introth:upic-hom-space}
we derive the main result of \cite{BvH07}, i.e the isomorphism
\eqref{eq:thm-old}
for  a principal homogeneous space $X$ of a connected $k$-group $G$.
Thus we obtain a new, more conceptual,
and less computational proof of this result.

We mention some applications and further results arising out of the main theorem.
\bigskip

\noindent{\bf Picard and Brauer groups}
\nopagebreak
\begin{introtheorem}[Theorem~\ref{cor:Pic-hom-space}]\label{intro-thm-Pic}
Let $X$ be a homogeneous space under a connected $k$-group $G$ with
  $\Pic(\Gbar) = 0$. Let $\Hbar$ be
the  stabilizer of a geometric point  $\xb \in X(\kbar)$
(we do not assume that $\Hbar$ is connected).
Then there is a canonical injection
$$
\Pic(X)\into H^1(k,[\XX(\Gbar)\to\XX(\Hbar)\rangle ),
$$
which is an isomorphism if $X(k)\neq\emptyset$ or $\Br(k)=0$.
\end{introtheorem}

Here $H^1$ denotes the first hypercohomology.

\begin{introtheorem}[Theorem~\ref{cor:br-hom-space}]\label{introth-Bra}
Let $X$, $G$, and $\Hbar$ be as in Theorem \ref{intro-thm-Pic}.
 Then  there is a canonical injection
  \begin{equation*}
  \Bra(X) \into H^2(k, [\Chi(\Gbar) \to \Chi(\Hbar) \rangle ),
  \end{equation*}
  which is an isomorphism if $X(k) \neq \emptyset$ or
  $H^3(k,\Gm) = 0$ (e.g., when $k$ is a number field or a
  $\mathfrak{p}$-adic field).
\end{introtheorem}

This proves the conjecture
\cite[Conj.~3.2]{Borovoi:hasse-homogeneous} of the first-named author
  about the subquotient $\Bra(X) := \ker [ \Br(X) \to
  \Br(\Xbar)] / \im [\Br(k) \to \Br(X)]$ of the Brauer group $\Br(X)$ of a
  homogeneous space $X$.

\begin{Remarks*}\label{rem:Demarche}
(1)  C.~Demarche  \cite{Demarche} computed the group $\Bra(X,G)$ (introduced in \cite{BD})
for a homogeneous space $X$ with \emph{connected} geometric stabilizers
of a connected $k$-group $G$ with $\Pic(\Gbar) = 0$,
when  $X(k) \neq \emptyset$ or  $H^3(k,\Gm) = 0$.
Here $\Bra(X,G)=\Br_1(X,G)/\Br(k)$, where $\Br_1(X,G)$ is the kernel of the composed homomorphism
$\Br(X)\to\Br(\Xbar)\to\Br(\Gbar)$.

(2)  In \cite{Borovoi-vanishing} the first-named author uses Theorem \ref{introth-Bra} of the present paper
in order to find sufficient conditions for the Hasse principle and weak approximation
for a homogeneous space of a \emph{quasi-trivial} $k$-group over a number field $k$,
with connected or abelian geometric stabilizer $\Hbar$,
in terms of the Galois module $\XX(\Hbar)$,
see \cite[Corollaries 2.2 and 2.6]{Borovoi-vanishing}.
As a consequence, he finds sufficient conditions for the Hasse principle and weak approximation
for  \emph{principal} homogeneous spaces of a connected linear algebraic $k$-group $G$
(where $G$ is not assumed to be quasi-trivial or such that $\Pic(\Gbar)=0$)
in terms of the Galois module $\pi_1(\Gbar)$,
see \cite[Corollaries 2.9 and 2.10]{Borovoi-vanishing}.
Here $\pi_1(\Gbar)$ is the algebraic fundamental group of $G$ introduced in \cite{Borovoi:Memoir}.
\end{Remarks*}

\bigskip
\noindent{\bf Comparison with topological invariants.}
In \cite{BvH07} the complex $\UPic(\Gbar)$ of a connected $k$-group $G$ was shown to be
dual (in the derived sense) to the {algebraic fundamental group}
of $G$  introduced in \cite{Borovoi:Memoir}.
In particular, if we fix an embedding $\kbar \into \Cc$
and an isomorphism $\pi_1(\Cc^\times)\isoto\Zz$,
then we have a canonical isomorphism in the derived category
$$
\UPic(\Gbar)^D\iso\pi_1(G(\Cc)),
$$
where we denote by $\UPic(\Gbar)^D$ the dual object
(in the derived sense) to $\UPic(\Gbar)$.
Also for homogeneous spaces we  have a topological interpretation of
$\UPic(\Xbar)^D$.

\begin{introtheorem}[Theorem~\ref{th:upicd-p1p2}]\label{introth:topological}
    Let $X$ be a homogeneous space under a connected (linear) $k$-group $G$
with connected  geometric stabilizers.
Let us fix an embedding $\kbar \into \Cc$
and an isomorphism $\pi_1(\Cc^\times)\isoto\Zz$.
  \begin{theoremlist}
    \item
      We have a canonical isomorphism of  groups
      \begin{equation*}
         \sH^0(\UPic(\Xbar)^D)  \iso \pi_1(X(\Cc)).
      \end{equation*}
     \item
       We have a canonical isomorphism of abelian groups
       \begin{equation*}
         \sH^{-1}(\UPic(\Xbar)^D)  \iso \pi_2(X(\Cc))/\pi_2(X(\Cc))\tors,
       \end{equation*}
       where $\pi_2(X(\Cc))\tors$ is the torsion subgroup of $\pi_2(X(\Cc))$.
  \end{theoremlist}
\end{introtheorem}

\begin{Remarks*}
 (1) We see from Theorem \ref{introth:topological}(i) that for a homogeneous space $X$
with connected geometric stabilizers, the topological fundamental group $\pi_1(X(\Cc))$
has a canonical structure of a Galois module.
This is closely related to the fact that any element of  $\Hom(\pi_1(\Cc^\times), \pi_1(X(\Cc))$ is \emph{algebraic},
i.e. can be represented by a regular map ${\mathb{G}}_{{\rm m},\kbar}\to\Xbar$, cf. \cite[Rem.~1.14]{Borovoi:Memoir}.

\item
(2) The assumption in Theorem \ref{introth:topological}(ii) that the geometric stabilizers are connected
can be somewhat relaxed, see Proposition \ref{prop:upicd-p1p2-H1},
but some condition must definitely be imposed, see Example \ref{ex:counter-ex-H1}.
\end{Remarks*}

\bigskip
\noindent\textbf{The elementary obstruction.}
In \cite{BvH07} it was shown that for any smooth geometrically integral $k$-variety $X$
the \emph{elementary obstruction}
(to the existence of a $k$-point in $X$)
as defined by Colliot-Th\'el\`ene and Sansuc \cite[D\'ef.\ 2.2.1]{CT-Sansuc}
(see also \cite[Introduction]{BCTS07})
may be identified with
a class $e(X)\in\Ext^1(\UPic(\Xbar),\kbar^\times)$
naturally arising from the construction of $\UPic(\Xbar)$.
With $X$, $G$ and $\Hbar$ as above,
  we have for every integer $ i$ a canonical isomorphism
  \begin{equation*}
    \Ext^i(\UPic(\Xbar), \kbar^\times) \iso
    H^i\left(k, \langle H\mm(\kbar) \to G\tor(\kbar)]\right)
  \end{equation*}
(Corollary \ref{cor:comp-ext1-upic-gm}).
Here $G\tor$ and $H\mm$ are the $k$-groups of multiplicative type
with character groups $\Chi(\Gbar)$ and $\Chi(\Hbar)$, resp.,
and  $\langle H\mm(\kbar) \to G\tor(\kbar)]$ is the complex
of Galois modules with $H\mm(\kbar)$ in degree $-1$ and $G\tor(\kbar)$ in degree 0,
dual to the complex $[\XX(\Gbar)\to\XX(\Hbar)\rangle$.

The complex $\langle H\mm(\kbar) \to G\tor(\kbar)]$ already appeared
in earlier work by the first-named author.
In \cite{Borovoi:hasse-homogeneous} the first-named author defined (by
means of explicit cocycles) an obstruction class $\eta(G,X) \in
H^1\left(k, \langle H\mm(\kbar) \to G\tor(\kbar)]\right)$ to the existence of a
rational point on $X$,
see also \ref{subsec:cohomological-obstruction} below.
The results of the present paper enable us to show that this obstruction class
$\eta(G,X)$ coincides, up to sign, with the elementary obstruction.

\begin{introtheorem}[Theorem~{\ref{thm:eta(G,X)}}]\label{introth:compare-e-eta}
   Let $X$ be a homogeneous space under a connected (linear) $k$-group $G$ with
 $\Pic(\Gbar) = 0$.
Let $\Hbar$ be the  stabilizer of a geometric point  $\xb \in X(\kbar)$
(we do not assume that $\Hbar$ is connected).
Then $e(X)\in\Ext^1(\UPic(\Xbar),\kbar^\times)$ coincides with
$-\eta(G,X) \in  H^1\left(k, \langle H\mm \to G\tor]\right)$ under the canonical
identification
$$
 \Ext^1(\UPic(\Xbar), \kbar^\times) \iso
    H^1(k, \langle H\mm \to G\tor]).
$$
\end{introtheorem}

The results of the present paper were partially announced in the research announcement \cite{BvH06}.

\bigskip
\noindent\emph{Acknowledgements.}
The authors are very grateful to Cyril Demarche for carefully reading the paper and making helpful comments,
and to Joseph Bernstein and Ofer Gabber for proving Lemma \ref{lem:complexes}.
We are grateful to Pierre Deligne, Gerd Faltings, Daniel Huybrechts, and Thomas Zink for most useful discussions.
We  thank the referee for a thorough review and useful remarks.
The first-named author worked on the paper, in particular, while visiting the Max-Planck-Institut f\"ur Mathematik, Bonn;
he thanks the Institute for hospitality, support, and excellent working conditions.

\section{Preliminaries}\label{sec:preliminaries}

\begin{subsec}\label{subsec:UPic}
Let $X$ be a geometrically integral algebraic variety over a field $k$ of characteristic 0.
We denote by $\sO(X)$ the ring of regular functions on $X$,
and by $\sK(X)$ the field of rational functions on $X$.
We denote by $\OX$ and $\KX$ the corresponding multiplicative groups.
We denote by $\Div(X)$ the group of Cartier divisors on $X$,
and by $\Pic(X)$ the Picard group of $X$
(i.e. the group of isomorphism classes of invertible sheaves on $X$).
We have an exact sequence
$$
0\to \OX/k^\times\to\KX/k^\times\labelto{\divisor}\Div(X)\to\Pic(X)\to 0.
$$
We set
$$
U(X)=\OX/k^\times.
$$

We set $\Xbar=X\times_k \kbar$.
In~\cite{BvH07} we defined an object $\UPic(\Xbar)$ in the derived
category of discrete $\Gal(\kbar/k)$-modules with the property that
$$\sH^0(\UPic(\Xbar))=U(\Xbar),\ \sH^1(\UPic(\Xbar))=\Pic(\Xbar)
,\ \sH^i(\UPic(\Xbar))  = 0 \text{\ if $i \neq 0,1$.}
$$
This object is given by the complex
\begin{equation}\label{eq:upic-repr}
\KXbar/\kbar^\times \labelto{\divisor} \Div(\Xbar)
\end{equation}
with $\KXbar/\kbar^\times$ in degree 0 and $\Div(\Xbar)$ in degree 1.
(cf. \cite[Cor.~2.5, Rem.~2.6]{BvH07}).
\end{subsec}

\begin{subsec}\label{subsec:G}
Let $X$ be a geometrically integral
$k$-variety,
$G$ a connected linear algebraic $k$-group,
 and  $w\colon X\times_k G\to X$ a right action of $G$ on $X$:
$(x,g)\mapsto xg$, where $x\in X,\ g\in G$.
For $g \in G(\kbar)$ we write $w_g\colon \Xbar \to \Xbar$ for
the map $x \mapsto xg$.
We denote by $\XX(G)$ the character group of $G$,
i.e. $\XX(G)=\Hom_k(G,\Gm)$.

We denote by $\Pic_G(X)$ the group of isomorphism classes of
$G$-equivariant invertible sheaves $(\sL,\beta)$ on $X$,
where $\sL$ is an invertible
sheaf on $X$ and $\beta$ is a \emph{$G$-linearization} of $\sL$,
see  \cite[Ch.~1, \S3, Def.~1.6]{Mumford}   or Definition \ref{def:lin-beta}   below.
We say that two $G$-linearizations $\beta$ and $\beta'$ of an
invertible sheaf $\sL$ are \emph{equivalent} if $(\sL, \beta)$ and
$(\sL,\beta')$ are $G$-equivariantly isomorphic.
The group structure on $\Pic_G(X)$ is given by the tensor product.
We have a canonical homomorphism of abelian groups
$$
\Pic_G(X)\to\Pic(X)
$$
taking the class $[\sL,\beta]$ of $(\sL,\beta)$ to the class $[\sL]$ of $\sL$.

Note that the set of $G$-linearizations of an invertible sheaf  may be empty.
In other words, the homomorphism $\Pic_G(X)\to\Pic(X)$ need not be surjective.
The set of equivalence classes
of $G$-linearizations of an invertible sheaf is either empty or a coset in
$\Pic_G(X)$ of the subgroup of equivalence classes of
$G$-linearizations of the trivial invertible sheaf $\sO_X$
(i.e. a coset of the kernel of the map $\Pic_G(X)\to\Pic(X)$).

We define
\begin{align*}
\sO_G(X) & = \sO(X)^G, \\
U_G(X)& =(\sO(X)^G)^\times/k^\times,
\end{align*}
where $\sO(X)^G$ denotes the ring of $G$-invariant regular functions
on $X$.
\end{subsec}

\begin{subsec}
  {\it Cones and fibres.}
Let $f\fcolon P \to Q$ be
  a morphism of complexes of objects of an abelian category $\Acat$.
  We denote by $\langle P \to Q ]$ the \emph{cone} of $f$,
and by $[ P \to Q \rangle$ the {\em fibre} (or co-cone) of $f$, see \cite[\S\,1.1]{BvH07}.
Note that $[P\to Q\rangle = \langle P\to Q][-1]$.
Note also that if $f\colon P\to Q$ is a morphism of \emph{objects} of our abelian category,
then $[ P \to Q \rangle$ is the complex $P\labelto{f}Q$ with $P$ in degree 0,
and $\langle P \to Q ]$ is the complex $P\labelto{-f}Q$ with $Q$ in degree 0.
\end{subsec}

\section{The extended  equivariant Picard complex}
  \label{sec:extend-equiv-picard}

\begin{subsec}\label{subsec:calg-defs}
  Let $X$ and $G$ be as in \S\ref{subsec:G}.  For $n \geq 0$
  we write $\Calg^n(\Gbar, \OXbar)$ for the Galois module $\sO(\Xbar\times\Gbar^n)^\times$
of invertible regular functions on $\Xbar\times \Gbar^n$.
   For $f \in \Calg^n(\Gbar, \OXbar)$ we write
  $f_{g_1, \cdots, g_n} = f|_{X\times{(g_1, \dots, g_n)}} \in \OXbar$
  for the restriction to the fibre over  $(g_1, \dots, g_n) \in
  G^n(\kbar)$. As the notation suggests, we regard $\Calg^n(\Gbar, \OXbar)$
  as the group of \emph{algebraic $n$-cochains of $G$ with coefficients in $\OXbar$}.
The assignment
  \begin{equation*}
    f \mapsto (\ (g_1, \dots, g_n) \mapsto f_{g_1,\cdots, g_n}\ )
  \end{equation*}
  defines a canonical injection of $\Calg^n(\Gbar, \OXbar)$ into the group
  $C^n(G(\kbar), \OXbar)$ of ordinary cochains of the abstract group $G(\kbar)$
  with coefficients in $\OXbar$.
The usual construction gives rise to a cochain complex
 \begin{equation*}
   \Calg^0(\Gbar, \OXbar) \labelto{d^0} \Calg^1(\Gbar, \OXbar) \labelto{d^1}
   \Calg^2(\Gbar, \OXbar) \labelto{d^2} \cdots
 \end{equation*}
 (which we may identify with a subcomplex of the standard cochain
 complex \newline $C^n(G(\kbar), \OXbar)$).  We denote the corresponding
 cocycles by $\Zalg^*(\Gbar, \OXbar)$, the coboundaries by
 $\Balg^*(\Gbar, \OXbar)$ and the cohomology by $\Halg^*(\Gbar, \OXbar)$.
 In fact, we shall mainly be interested in the first two degrees of the complex.
 Here $d^0 \colon \OXbar \to \Calg^1(\Gbar, \OXbar)$ is
 given by
$$
f \mapsto  (f \circ w)/ (f \circ p_X),
\text{ i.e. } f\mapsto (\ (x,g)\mapsto f(xg)/f(x)\ ),
$$
where $w\colon X\times_k G\to X$ is the action of $G$ on $X$,
and $p_X\colon X\times_k G\to X$ is the first projection.
Note that $\Zalg^1(\Gbar,\OXbar) \subset \Calg^1(\Gbar, \OXbar)$
consists of those invertible regular functions
 $c$ on $X\times G$ for which
\begin{equation}\label{eq:cocycle}
c_{g_1 g_2}(x)=c_{g_1}(x)\cdot c_{g_2}(xg_1).
\end{equation}

We also define the complexes $\Calg^*(\Gbar, \KXbar)$ and
$\Calg^*(\Gbar, \Div(\Xbar))$
by taking $\Calg^n(\Gbar, \KXbar)$
(resp.~$\Calg^*(\Gbar, \Div(\Xbar)$) to be the group of invertible
rational functions (resp.~divisors) on $X\times G^n$.
The differentials of the complexes are defined as above.
\end{subsec}

\begin{subsec}\label{subsec:KDiv}
We write $\KDiv(\Xbar)$ for the $2$-term complex
$\left[\sK(\Xbar)^\times\labelto{\divisor} \Div(\Xbar)\right\rangle$ (as in \cite{BvH07}).
Recall that $\sK(\Xbar)$ is in degree 0 and $\Div(\Xbar)$ is in degree 1.
 We define
$\Calg^*(\Gbar, \KDiv(\Xbar))$ to be the total complex associated to
the double complex
\[
\xymatrix{
\cdots                                           & \cdots               \\
{\Calg^1}(\Gbar,\KXbar)\ar[r]^{\divisor^1}\ar[u]_{d^1_\sK}
                          & \Calg^1(\Gbar,\Div(\Xbar))\ar[u]_{d^1_{\Div}}  \\
{\Calg^0}(\Gbar,\KXbar)\ar[r]^{\divisor^0}\ar[u]_{d^0_\sK}
                          & \Calg^0(\Gbar,\Div(\Xbar))\ar[u]_{d^0_{\Div}}.
}
\]
In other words, $\Calg^*(\Gbar, \KDiv(\Xbar))$ is the total complex associated to
the double complex
$$
\xymatrix{
\cdots                                           & \cdots               \\
\sK(\Xbar\times\Gbar)^\times   \ar[r]^{\divisor^1}\ar[u]_{d^1_\sK}
                          & \Div(\Xbar\times\Gbar)   \ar[u]_{d^1_{\Div}}  \\
\KXbar                         \ar[r]^{\divisor^0}\ar[u]_{d^0_\sK}
                          & \Div(\Xbar)\ar[u]_{d^0_{\Div}} .
}
$$
Here $\KXbar$ is in bidegree $(0,0)$.
We write $\Halg^i(\Gbar, \KDiv(\Xbar))$ for \\$\sH^i(\Calg^*(\Gbar, \KDiv(\Xbar)))$.
\end{subsec}

\begin{definition}\label{def:UPicG}
Let $X$ and $G$ be as in \S\ref{subsec:G}.
We define  the \emph{extended equivariant Picard complex of  $G$ and $X$} to be the complex
\begin{equation*}
  \UPic_G(\Xbar) = \tau_{\leq 1} \Calg^*(\Gbar, \KDiv(\Xbar))/\kbar^\times.
\end{equation*}
\end{definition}

In other words, $\UPic_G(\Xbar)$ is the complex
\begin{equation*}
  \KXbar/\kbar^\times
    \labelto{\begin{pmatrix}
                 d^0_\sK\rule[-2ex]{0pt}{2ex} \\ \divisor^0
             \end{pmatrix}}
  \left\{ (z, D) \in \Zalg^1(\Gbar, \KXbar) \oplus \Div(\Xbar) \colon
          \divisor^1(z) = d^0_{\Div}(D) \right\},
\end{equation*}
where $d^0_\sK([f]) = (f \circ w)/(f \circ p_X)$ and
$d^0_{\Div}(D) = w^*D-p_X^*D$.
Clearly we have
$$
\sH^0(\UPic_G(\Xbar))=\Halg^0(\Gbar, \KDiv(\Xbar))/\kbar^\times,
\quad \sH^1(\UPic_G(\Xbar))=\Halg^1(\Gbar, \KDiv(\Xbar)).
$$
We have
$$
\Halg^0(\Gbar, \KDiv(\Xbar))
  = \{ f \in \KXbar \colon \divisor(f) = 0 \text{ and } f \circ w =
  f \circ p_X \}=(\sO(\Xbar)^\times)^G,
$$
whence
$$
\sH^0(\UPic_G(\Xbar))=(\sO(\Xbar)^\times)^G/\kbar^\times=U_G(\Xbar).
$$
It turns out that
$$
 \sH^1(\UPic_G(\Xbar))=\Pic_G(\Xbar),
$$
see  Corollary~\ref{cor:coh-upicg} below.

Note that by slight abuse of notation we  write
$\UPic_G(\Xbar)$, $U_G(\Xbar)$, $\Pic_G(\Xbar)$ rather than
$\UPic_{\Gbar}(\Xbar)$, $U_{\Gbar}(\Xbar)$, $\Pic_{\Gbar}(\Xbar)$.

\begin{subsec}\label{subsec:upicg-forgetful-map}
The complex $\UPic_G(\Xbar)$ can be regarded  as an equivariant version
of the extended Picard complex
$\UPic(\Xbar)=\left[\KXbar/\kbar^\times
\labelto{\divisor}\Div(\Xbar)\right\rangle$
(indeed, if $G=1$, then $\Zalg^1(\Gbar, \KXbar)=1$ and $\UPic_G(\Xbar)=\UPic(\Xbar)$).
We have an obvious natural morphism of complexes of Galois modules
\begin{equation*}
\nu\colon\UPic_G(\Xbar)\to\UPic(\Xbar),
\end{equation*}
where $\nu^0=\id$ and
$$
\nu^1(z,D)=D
\text{ for }(z,D)\in\UPic_G(\Xbar)^1\subset \Zalg^1(\Gbar,\sK(\Xbar)^\times)\oplus\Div(\Xbar).
$$
\end{subsec}

\begin{subsec}\label{subsec:functoriality}
{\em Functoriality.}

Let $X$ and $G$ be as in \S\ref{subsec:G}.  It is clear that a
homomorphism $G' \to G$ of linear algebraic groups over $k$ induces a
pull-back homomorphism
\begin{equation*}
 \Calg^*(\Gbar, \hvar) \to \Calg^*(\Gbar', \hvar)
\end{equation*}
for any of the coefficients considered,
and also a pull-back homomorphism
\begin{equation*}
  \UPic_G(\Xbar) \to \UPic_{G'}(\Xbar).
\end{equation*}

Functoriality in $X$ is a bit more subtle.
For a \emph{dominant} $G$-equivariant morphism
\begin{equation*}
  f \colon X' \to X,
\end{equation*}
from another $G$-variety $X'$ as in Section~\ref{subsec:G} to $X$, we
clearly have a pull-back morphism of complexes
$\UPic(\Xbar) \to \UPic(\Xbar')$
and a pull-back morphism of complexes $\UPic_G(\Xbar) \to \UPic_G(\Xbar')$.

However, for a $G$-equivariant morphism $f$ as above that is not
dominant, we need to modify our complexes.
We assume that both $X$ and $X'$ are smooth.
We choose an arbitrary $G$-invariant scheme-theoretic point $x' \in X'$,
which need not be closed.
In more geometric terms this amounts to taking the
generic point of an irreducible, but not necessarily geometrically
irreducible $G$-orbit on $X'$.  Taking $x = f(x')$, we then consider
the subcomplex
$\OxDiv(\Xbar) \subset \KDiv(\Xbar)$
given by
\begin{equation*}
  (\sO_{X,x}\tensor_k \kbar)^\times  \labelto{\divisor} \Div(\Xbar)_x\, ,
\end{equation*}
where $\sO_{X,x} \subset \sK(X)$ is the local ring at $x$ and $\Div(\Xbar)_x$
consists of the divisors whose support does not contain $x$.
Using a moving lemma for divisors on a smooth variety
(cf. \cite{Shafarevich}, Vol.~1, III.1.3, Thm.~1 and the remark after the proof),
we see that the
inclusion $\OxDiv(\Xbar) \into \KDiv(\Xbar)$ is a quasi-isomorphism, and so is
the induced inclusion $\OxDiv(\Xbar)/\kbar^\times \into \UPic(\Xbar)$. We shall
denote $\OxDiv(\Xbar)/\kbar^\times$ by $\UPic(\Xbar)_x$. Similarly, we define
\begin{equation*}
  \UPic_G(\Xbar)_x = \tau_{\le 1} \Calg^*(\Gbar, \OxDiv(\Xbar))/\kbar^\times
\end{equation*}
and we see that the canonical inclusion $\UPic_G(\Xbar)_x \into
\UPic_G(\Xbar)$ is a quasi-isomorphism.

By construction, a $G$-equivariant morphism $f$ as above now induces
pull-back homomorphisms of complexes
\begin{align*}
  f^*_x \colon \UPic(\Xbar)_x & \to \UPic(\Xbar')_{x'} \\
\intertext{and}
  f^*_x \colon \UPic_G(\Xbar)_x & \to \UPic_G(\Xbar')_{x'},
\end{align*}
hence morphisms in the derived category
\begin{align}
  f^* \colon \UPic(\Xbar) & \to \UPic(\Xbar') \notag \\
\intertext{and}
  f^* \colon \UPic_G(\Xbar) & \to \UPic_G(\Xbar'). \label{eq:f*G}
\end{align}
For $\UPic(\hvar)$ this morphism coincides with the morphism we obtain
from the derived functor construction of $\UPic$ in \cite[2.1]{BvH07}, so it
is independent of the choice of $x' \in X'$.  In the case of $\UPic_G$
we shall have to verify this directly by an auxiliary construction.

For this, we generalize the above construction from a point $x \in X$
to a finite set $S = \{x_1, \dots, x_n\}$ of $G$-equivariant points in
$X$ by replacing $\sO_{X,x} \otimes \kbar$ by
\[ \sO_{X,S}\otimes \kbar =
\bigcap_{x \in S} \sO_{X, x} \otimes \kbar \subset \sK(\Xbar) \]
and replacing $\Div(\Xbar)_x$ by
\[ \Div(\Xbar)_S = \cap_{x \in S} \Div(\Xbar)_x \subset \Div(\Xbar), \]
etc.  We then see from the diagram below, in which every injection of
complexes is a quasi-isomorphism, that the morphism~(\ref{eq:f*G}) in
the derived category does not depend on the choice of $x' \in X'$:
\begin{equation*}
\xymatrix@!R=2ex@C=0pt@M=1.1ex{
 & \UPic_G(\Xbar)_{x_1} \ar[rr]^{f^*_{x_1}} && \UPic_G(\Xbar')_{x'_1} \\
\UPic_G(\Xbar)_{\{x_1, x_2\}} \ar[rr]^{f^*_S} \ar@{^{(}->}[dr]\ar@{^{(}->}[ur]
&& \UPic_G(\Xbar')_{\{x'_1, x'_2\}} \ar@{^{(}->}[dr]\ar@{^{(}->}[ur]  \\
 & \UPic_G(\Xbar)_{x_2} \ar[rr]^{f^*_{x_2}} && \UPic_G(\Xbar')_{x'_2}. \\
}
\end{equation*}
We shall call any modification where we replace $\KXbar$ by
$(\sO_{X,x}\otimes_k\kbar)^\times$ and $\Div(\Xbar)$ by $\Div(\Xbar)_x$, etc.  a
\emph{local modification}.
\end{subsec}


\section{ $G$-linearizations of invertible sheaves}\label{sec:Glin}

\begin{subsec}
Let $G$ be an algebraic group (not necessarily linear)
over an algebraically closed field $k$ of characteristic 0.
Let $X$ be a $k$-variety with a right action $w$ of $G$.
This means that we are given a morphism of varieties
$$
w\colon X\times G\to X,\quad (x,g)\mapsto w_g(x)=xg,
$$
satisfying the usual conditions.
\end{subsec}

\begin{definition}[{\cite[Ch.~1, \S3, Def.~1.6]{Mumford}}]
\label{def:lin-beta}
Let $\sL$ be an invertible sheaf on a $G$-variety $X$.
A $G$-linearization of $\sL$ is an isomorphism
$$
\beta \colon w^*\sL\to p^*_X \sL
$$
of invertible sheaves on $X\times G$
satisfying the following cocycle condition:

 Let $m\colon G\times G\to G$ be the group law.
Consider the projection $p_X\colon X\times G\to X$.
We have morphisms $w$ and $p_X$ from $X\times G$ to $X$.
Consider the projection
$p_{X,G1}\colon X\times G\times G\to X\times G$
taking $(x,g_1,g_2)$ to $(x,g_1)$.
We have morphisms $p_{X,G1},\ 1_X\times m$, and $w\times 1_G$
from $X\times G\times G$ to $X\times G$.
The cocycle condition is the commutativity of the following diagram:
\begin{equation*}
\xymatrix{
[w\circ(w\times 1_G)]^* \sL \ar@{=}[dd] \ar[rr]^{(w\times 1_G)^*\beta} &
        &[p_X\circ(w\times 1_G)]^*\sL\ar@{=}[d]\\
& &[w\circ p_{X,G1}]^*\sL\ar[r]^{(p_{X,G1})^*\beta}        &[p_X\circ p_{X,G_1}]^*\sL\ar@{=}[d] \\
[w\circ(1_X\times m)]^*\sL\ar[rrr]^{(1_X\times m)^*\beta} & & &[p_X\circ(1_X\times m)]^*\sL .
}
\end{equation*}
This is the same as to say
that for each $g_1,g_2\in G(k)$ we have the cocycle condition
\begin{equation}\label{eq:cocycle-beta}
\beta_{g_1 g_2}=\beta_{g_1}\circ w^*_{g_1}(\beta_{g_2}),
\end{equation}
where for $g\in G(k)$ we write $\beta_g$
for the inverse image of $\beta$ under the map $X\to X\times G,\ x\mapsto (x,g)$.
\end{definition}

\begin{definition}
 Let $G$ be a $k$-group and $X$ be a $G$-variety over $k$.
By an invertible $G$-sheaf we mean a pair $(\sL,\beta)$,
where $\sL$ is an invertible sheaf and $\beta$
is a $G$-linearization of $\sL$.
We denote by $\Pic_G(X)$ the group of isomorphism classes
of invertible $G$-sheaves on a $G$-variety $X$.
We denote by $[\sL,\beta]$ the class of the pair $(\sL,\beta)$ in $\Pic_G(X)$.
\end{definition}

We wish to compute this group $\Pic_G(X)$ in terms
of divisors and rational functions (see Theorem \ref{thm:PicG-H1}   below).

\begin{subsec}
From now on we assume that $G$ is a \emph{connected linear} $k$-group
and $X$ is an \emph{integral} $G$-variety.
We denote by $\sO_X$ the structure sheaf on $X$, and by $\sK_X$
the sheaf of total quotient rings of $\sO_X$.
Then $\sO(X)=\Gamma(X,\sO_X)$ and $\sK(X)=\Gamma(X,\sK_X)$.

By an invertible $\sK_X$-sheaf $\sR$ on $X$ we mean a locally free sheaf of modules
of rank one over the sheaf of rings $\sK_X$.
Note that if an invertible  $\sK_X$-sheaf $\sR$ has a non-zero global section $s$, then
$\sR$ is isomorphic to $\sK_X$ as a $\sK_X$-module.

Let $\sL$ be an invertible sheaf on $X$.
We set $\sL^\sK=\sL\otimes_{\sO_X} \sK_X$.
Note that $\sL^\sK$ has a non-zero global section, because $\sL$ has a non-zero rational section.
If $\psi\colon \sL_1\to\sL_2$ is a morphism of invertible sheaves on $X$,
then we have an induced morphism
$$
\psi^\sK\colon \sL_1^\sK\to\sL_2^\sK .
$$

Let $f\colon X\to Y$ be a \emph{dominant} morphism of integral $k$-varieties.
Then we have a morphism of ringed spaces
$$
(X,\sK_X)\to (Y,\sK_Y).
$$
If $\sR_Y$ is a sheaf of $\sK_Y$-modules on $Y$,
then we define
\begin{equation*}
f^*\sR_Y=f^{-1}\sR_Y\otimes_{f^{-1}\sK_Y}\sK_X,
\end{equation*}
cf. \cite[Ch.~II, Section 5, p.~110]{Hartshorne}.
If $\sL_Y$ is a sheaf of $\sO_Y$-modules on $Y$, then
\begin{equation*}
f^*(\sL_Y\otimes_{\sO_Y}\sK_Y)=(f^*\sL_Y)\otimes_{\sO_X}\sK_X.
\end{equation*}
\end{subsec}

\begin{definition}\label{def:lin-gamma}
A $G$-linearization of an invertible $\sK_X$-sheaf $\sR$
on a $G$-variety $X$ is an isomorphism of invertible $\sK_{X\times G}$-sheaves on $X\times G$
$$
\gamma\colon w^*\sR\to p_X^*\sR
$$
such that the following diagram commutes:
\begin{equation*}
\xymatrix{
[w\circ(w\times 1_G)]^* \sR \ar@{=}[dd] \ar[rr]^{(w\times 1_G)^*\gamma} &
        &[p_X\circ(w\times 1_G)]^*\sR\ar@{=}[d]\\
& &[w\circ p_{X,G1}]^*\sR\ar[r]^{(p_{X,G1})^*\gamma}        &[p_X\circ p_{X,G_1}]^*\sR\ar@{=}[d] \\
[w\circ(1_X\times m)]^*\sR\ar[rrr]^{(1_X\times m)^*\gamma} & & &[p_X\circ(1_X\times m)]^*\sR .
}
\end{equation*}
\end{definition}

Note that the diagram of Definition \ref{def:lin-gamma}  is just the diagram of Definition \ref{def:lin-beta}
with $\sR$ instead of $\sL$ and with $\gamma$ instead of $\beta$.

\begin{lemma}\label{lem:G-lin-L-R}
Let $\sL$ be an invertible sheaf on $X$, and
 let $\beta\colon w^*\sL\to p_X^*\sL$ be an isomorphism.
Then $\beta$ is a $G$-linearization of $\sL$ if and only if
$$
\beta^\sK\colon w^*\sL^\sK\to p^*_X\sL^\sK
$$
is a $G$-linearization of $\sL^\sK$.
\end{lemma}

\begin{proof}
It is clear that if $\beta$ is a $G$-linearization of an invertible sheaf $\sL$ on $X$,
then
$$
\beta^\sK\colon w^*\sL^\sK\to p^*_X\sL^\sK
$$
is a $G$-linearization of $\sL^\sK$.
Conversely, assume that $\gamma:=\beta^\sK$ is a $G$-linearization of $\sR:=\sL^\sK$.
We compare the isomorphisms
$$
(1_X\times m)^*\beta,\  (p_{X,G1})^*\beta\circ (w\times 1_G)^*\beta\colon\
[w\circ(1_X\times m)]^*\sL \to [p_X\circ(1_X\times m)]^*\sL
$$
from the diagram of Definition \ref{def:lin-beta}.
We may write
$$
(1_X\times m)^*\beta=\psi\cdot (p_{X,G1})^*\beta\circ (w\times 1_G)^*\beta
$$
for some $\psi\in\sO(X\times G\times G)^\times$.
We substitute $\sR=\sL^\sK$ and  $\gamma =\beta^\sK$ in the diagram of Definition \ref{def:lin-gamma},
and we obtain that
$$
(1_X\times m)^*\gamma=\psi\cdot (p_{X,G1})^*\gamma\circ (w\times 1_G)^*\gamma.
$$
But by assumption $\gamma$ is a $G$-linearization of $\sL^\sK$,
hence $\gamma$ makes commutative the diagram of Definition \ref{def:lin-gamma},
i.e.
$$
(1_X\times m)^*\gamma= (p_{X,G1})^*\gamma\circ (w\times 1_G)^*\gamma.
$$
We see that $\psi=1$, therefore
$$
(1_X\times m)^*\beta= (p_{X,G1})^*\beta\circ (w\times 1_G)^*\beta,
$$
hence $\beta$ makes commutative the diagram of Definition \ref{def:lin-beta}.
Thus $\beta$ is a $G$-linearization of $\sL$.
\end{proof}

\begin{definition}
Let $X$ be a $G$-variety.
As in \ref{subsec:calg-defs},
we define $Z^1\alg(G,\sK(X)^\times)$
to be the group of nonzero rational functions $z\in\sK(X\times G)^\times$
satisfying the cocycle condition
\begin{equation}\label{eq:cocycle-z}
(1_X\times m)^* z=(p_{X,G1})^* z\cdot (w\times 1_G)^* z.
\end{equation}
Of course, this is the same as to say that
$$
z_{g_1 g_2}(x)=z_{g_1}(x)\cdot z_{g_2}(x g_1)
$$
for all triples $(x,g_1,g_2)\in X(k)\times G(k)\times G(k)$
for which all the three values $z_{g_1 g_2}(x),\ z_{g_1}(x)$, and $z_{g_2}(x g_1)$
are different from 0 and $\infty$.
\end{definition}

\begin{lemma}\label{lem:K-linearizations}
Let $\sR$ be an invertible $\sK_X$-sheaf on a $G$-variety $X$ having a
nonzero section $s$.
Then there exists a canonical bijection
\begin{align*}
\{ G\text{-linearizations of }\sR\}   &\to    Z\alg^1(G,\sK(X)^\times)\\
\gamma\quad &\mapsto\quad z
\end{align*}
such that $\gamma(w^*s)=z\cdot p^*_X s$.
\end{lemma}

\begin{proof}
Let $\alpha\colon \sK_X\to\sR$ be the isomorphism of sheaves of $\sK_X$-modules
such that $\alpha(1)=s$, where $1\in\Gamma(X,\sK_X)=\sK(X)$ is the unit element.
The isomorphism
$$
\gamma\colon w^*\sR\to p^*_X\sR
$$
gives, via $\alpha$, an automorphism
$$
\gamma'\colon\sK_{X\times G}\to\sK_{X\times G}
$$
and this automorphism $\gamma'$ is given by multiplication by
a rational function $z\in\sK(X\times G)^\times$.
The cocycle condition of commutativity of the diagram of Definition \ref{def:lin-gamma} writes then as
\begin{equation*}
(1_X\times m)^*z=(p_{X,G1})^*z\cdot (w\times 1_G)^*z.
\end{equation*}
We see that $z\in Z\alg^1(G,\sK(X)^\times)$.
Note  that $\gamma'(1)=z$, and we can write it as  $\gamma'(w^*1)=z\cdot p^*_X 1$.
Returning to our original sheaf $\sR$ and section $s$,
we obtain that
$$
\gamma(w^*s)=z\cdot p^*_X s.
$$
It is easy to see that our map $\gamma\mapsto z$ is bijective.
\end{proof}

\begin{lemma}[cf. \cite{KKV}, 2.1] \label{lem:G-linearizations}
 Consider the trivial invertible sheaf $\sL=\sO_X$ on a $G$-variety $X$.
There exists a canonical isomorphism of abelian groups
\begin{align*}
\{ G\text{-linearizations of }\sO_X\} &\to    Z\alg^1(G,\sO(X)^\times)\\
\beta\quad &\mapsto\quad c
\end{align*}
such that $\beta(1)=c$.
\end{lemma}

\begin{proof}
 Similar to that of Lemma \ref{lem:K-linearizations}.
\end{proof}

\begin{subsec}
Let  $\sL_1,\sL_2$ be two invertible sheaves on a variety $X$,
and let $s_i$ be a nonzero rational section of $\sL_i$ ($i=1,2$).
Let $\gamma\colon \sL_1^\sK\to\sL_2^\sK$ be an isomorphism.
Then there exists a unique rational function $f_\gamma\in\sK(X)^\times$
such that
$$
\gamma(s_1)=f_\gamma\cdot s_2.
$$
Conversely, it is clear that for any $f\in\sK(X)^\times$
there exists an isomorphism $\gamma\colon \sL_1^\sK\to\sL_2^\sK$
such that $f_\gamma=f$.
\end{subsec}

\begin{lemma}\label{lem:from regular isomorphism}
An isomorphism $\gamma\colon \sL_1^\sK\to\sL_2^\sK$ as above
comes from some  isomorphism  of invertible sheaves $\beta\colon \sL_1\to\sL_2$
if and only if
$$
\divisor (f_\gamma)=\divisor(s_1) - \divisor(s_2).
$$
\end{lemma}

\begin{proof}
Assume that $\gamma=\beta^\sK$ for some isomorphism  (of sheaves of $\sO_X$-modules)
$\beta\colon \sL_1\to\sL_2$.
Then
$$
\beta(s_1)=f_\gamma\cdot s_2.
$$
Since $\beta$ is an isomorphism of sheaves of $\sO_X$-modules,
we have
$$
\divisor(\beta(s_1))=\divisor(s_1).
$$
Thus
$$
\divisor(f_\gamma)=\divisor(\beta(s_1))-\divisor(s_2)=\divisor(s_1)-\divisor(s_2).
$$

Conversely, assume that
$$
\divisor(f_\gamma)=\divisor(s_1)-\divisor(s_2).
$$
Then $\div(s_1)-\div(s_2)$ is a principal divisor, hence the classes of $\sL_1$ and $\sL_2$ in $\Pic(X)$ are equal.
It follows that there exists an isomorphism
$\beta'\colon \sL_1\to\sL_2$.
We obtain as above a rational function $f_{\beta'}$ such that
$$
\beta'(s_1)=f_{\beta'}\cdot s_2.
$$
As above, we have
$$
\divisor(f_{\beta'})=\divisor(s_1)-\divisor(s_2).
$$
We see that $\div(f_{\beta'})=\div(f_\gamma)$,
hence $f_\gamma=\varphi f_{\beta'}$ for some $\varphi\in\sO(X)^\times$.
Set $\beta=\varphi \beta'\colon\sL_1\to\sL_2$,
then $f_\beta=\varphi f_{\beta'}=f_\gamma$,
hence $\gamma=\beta^\sK$.
Thus $\gamma$ comes from the isomorphism $\beta$
of sheaves of $\sO_X$-modules.
\end{proof}

\begin{definition}
Let $X$ be a $G$-variety, where $G$ is connected and $X$ is integral.
We define
\begin{align*}
Z^1\alg&(G,\KDiv(X))=   \\
    &\{(z,D)\ |\ z\in Z\alg^1(G,\sK(X)),\ D\in\Div(X),\ \div(z)=w^*D-p_X^*D\}.
\end{align*}
We define a homomorphism
$$
d\colon \sK(X)^\times\to Z^1\alg(G,\KDiv(X)),\ f\mapsto (w^*(f)/p^*_X(f), \div(f)).
$$
We set $B^1\alg(G,\KDiv(X))=\im d$ and
$$
H\alg^1(G,\KDiv(X))=Z^1\alg(G,\KDiv(X))/B^1\alg(G,\KDiv(X)),
$$
as in \ref{subsec:KDiv}.
\end{definition}

\begin{theorem}\label{thm:PicG-H1}
There is a canonical isomorphism
$$
\Pic_G(X)\isoto H\alg^1(G,\KDiv(X)).
$$
\end{theorem}

\begin{proof}
We construct a map
$$
\varkappa\colon \Pic_G(X)\to H\alg^1(G,\KDiv(X)).
$$
Let $[\sL,\beta]\in\Pic_G(X)$, where $\sL$ is an invertible sheaf on $X$ and
$$
\beta\colon w^*\sL\to p_X^*\sL
$$
is a $G$-linearization.
Tensoring with $\sK_X$ we obtain a $G$-linearization of $\sL^\sK$
$$
\beta^\sK\colon w^*\sL^\sK\to p_X^*\sL^\sK.
$$
Choose a rational section $s$ of $\sL$, i.e. a section of $\sL^\sK$.
By Lemma \ref{lem:K-linearizations}  $\beta^\sK$
corresponds to a cocycle  $z\in Z\alg^1(G,\sK(X)^\times)$
such that
$$
\beta^\sK(w^*s)=z\cdot p^*_X s.
$$
Since $\beta^\sK$ comes from an isomorphism of sheaves of $\sO_{X\times G}$-modules
$\beta\colon w^*\sL\to p_X^*\sL$,
by Lemma \ref{lem:from regular isomorphism} we have
$$
\div(z)=\div(w^*s)-\div(p_X^*s).
$$
Set $D=\divisor(s)$, then
$$
\div(z)=w^*D-p_X^*D.
$$
We see that $(z,D)\in Z^1\alg(G,\KDiv(X))$.
We set
$$
\varkappa([\sL,\beta])=[z,D]\in H^1\alg(G,\KDiv(X)).
$$
An easy calculation shows that $\varkappa$ is a well defined homomorphism.

We construct a map
$$
\lambda\colon   H\alg^1(G,\KDiv(X))\to \Pic_G(X).
$$
Let $(z,D)\in Z\alg^1(G,\KDiv(X))$.
Then the divisor $D$ defines an invertible sheaf $\sL$ together with a nonzero rational
section $s$ of $\sL$ (i.e a section of $\sL^\sK$)
such that $\div(s)=D$.
By Lemma \ref{lem:K-linearizations} $z$ defines a $G$-linearization
$$
\gamma\colon w^*\sL^\sK\to p_X^*\sL^\sK
$$
such that
\begin{equation}\label{eq:gamma}
\gamma(w^*s)=z\cdot p_X^*s.
\end{equation}
Since $(z,D)\in Z\alg^1(G,\KDiv(X))$, we have $\div(z)=w^*D-p_X^*D$, hence
\begin{equation}\label{eq:div}
\div(z)=\div(w^*s)-\div(p_X^*s).
\end{equation}
By Lemma \ref{lem:from regular isomorphism} it follows from \eqref{eq:gamma} and \eqref{eq:div}
that $\gamma$ comes from an isomorphism of $\sO_{X\times G}$-modules
$$
\beta\colon w^*\sL\to p_X^*\sL,
$$
that is, $\gamma =\beta^\sK$.
Since $\gamma$ is a $G$-linearization of $\sL^\sK$,
by Lemma \ref{lem:G-lin-L-R} $\beta$ is a $G$-linearization of $\sL$.
We set
$$
\lambda([z,D])=[\sL,\beta]\in\Pic_G(X).
$$
Easy calculations show that $\lambda$ is a well defined homomorphism
and that $\varkappa$ and $\lambda$ are mutually inverse.
Thus $\varkappa$ is an isomorphism.
\end{proof}


\section{Relations between $\UPic_G(\Xbar)$ and $\UPic(\Xbar)$}
  \label{sec:rel-upic_g-upic}

We return to the assumptions and notation of \S\ref{subsec:G}.
Theorem \ref{thm:PicG-H1} says that
we have a canonical isomorphism of Galois modules
$$
  \Halg^1(\Gbar, \KDiv(\Xbar))  \cong \Pic_G(\Xbar).
$$

\begin{corollary}\label{cor:coh-upicg}
  Let $X$ and $G$ be as in \S\ref{subsec:G}.
 We have a canonical isomorphism of Galois modules
\[
  \sH^1(\UPic_G(\Xbar)) \cong \Pic_G(\Xbar).
\]
\end{corollary}
\begin{proof}
  This follows immediately from Theorem~\ref{thm:PicG-H1}, since
$\sH^1(\UPic_G(\Xbar)) = \Halg^1(\Gbar, \KDiv(\Xbar))$.
\end{proof}

\begin{question}
What is a geometric interpretation of the cohomology groups
  \\$\Halg^n(\Gbar, \KDiv(\Xbar))$  for $n > 1$?
\end{question}

\begin{lemma}[{\cite[Lemma~2.2]{KKV}}]
\label{cor:exact-seq-picg-pic-g-conn-trivpic}
  Let $X$ and $G$ be as in \S\ref{subsec:G} and assume that $X$ is normal.
Then the canonical homomorphism $\Pic_G(\Xbar) \to  \Pic(\Xbar)$ fits into an exact sequence
  \begin{equation}
    \label{eq:exact-seq-picg-pic-g-conn-trivpic}
        0 \to \Halg^1(\Gbar, \OXbar) \to \Pic_G(\Xbar) \to  \Pic(\Xbar) \to \Pic(\Gbar).
\end{equation}
\qed
\end{lemma}

From now on  we shall assume that the natural map $\Pic_G(\Xbar) \to \Pic(\Xbar)$ is surjective.
By Lemma \ref{cor:exact-seq-picg-pic-g-conn-trivpic} this condition is satisfied when $X$ is normal and $\Pic(\Gbar)=0$
(which  can be forced for homogeneous spaces of connected groups, as we shall see below).
Examples of $k$-groups that satisfy  $\Pic(\Gbar)=0$ are
algebraic tori, simply connected  semisimple groups, and quasi-trivial groups.

\begin{lemma}\label{lem:upicg-decomp}
  Let $X$ and $G$ be as in \S\ref{subsec:G}. If $\Pic_G(\Xbar) \to
  \Pic(\Xbar)$ is surjective, we have a short exact sequence of
  complexes
\begin{equation*}
  \label{eq:upicg-decomp}
  0 \to  \Zalg^1(\Gbar,\OXbar)[-1]  \labelto{\mu} \UPic_G(\Xbar) \labelto{\nu}
\UPic(\Xbar) \to 0,
\end{equation*}
which is functorial in $X$ and $G$.
Here $\nu$ is the morphism of \S\ref{subsec:upicg-forgetful-map}, and for $c\in \Zalg^1(\Gbar,\OXbar)$ we set
$$
\mu^1(c)=(c^{-1},0)\in \UPic_G(\Xbar)^1\subset \Zalg^1(\Gbar,\sK(\Xbar)^\times)\oplus\Div(\Xbar).
$$
\end{lemma}

\begin{proof}
Consider the canonical morphism of complexes
$\nu\colon\UPic_G(\Xbar) \to \UPic(\Xbar)$
of \S\ref{subsec:upicg-forgetful-map}.
An easy diagram chase in the commutative diagram with exact rows
$$
\xymatrix{
\sK(\Xbar)^\times/\kbar^\times \ar[r]\ar@{=}[d] &\UPic_G(\Xbar)^1 \ar[r]\ar[d]^{\nu^1} &\Pic_G(\Xbar) \ar[r]\ar[d] &0\\
\sK(\Xbar)^\times/\kbar^\times   \ar[r]     & \Div(\Xbar)     \ar[r]       &\Pic(\Xbar)   \ar[r]       &0
}
$$
shows that if the map $\Pic_G(\Xbar) \to \Pic(\Xbar)$ is surjective,
then the map $\nu^1\colon\UPic_G(\Xbar)^1\to \Div(\Xbar)$ is surjective
and hence the morphism of complexes $\nu\colon\UPic_G(\Xbar) \to \UPic(\Xbar)$ is surjective.
On the other hand, the kernel of the morphism  $\nu$
coincides with the complex
$\Zalg^1(\Gbar,\OXbar)[-1]=\im \mu$.
\end{proof}

In the rest of this section we derive consequences of Lemma \ref{lem:upicg-decomp}.
We consider a more general setting.

\begin{subsec}\label{subsec:complexes}
 Let
\begin{equation}\label{eq:complexes}
0\to A\updot \labelto{\mu} B\updot
                         \labelto{\nu} C\updot\to 0
\end{equation}
be a short exact sequence of complexes in an abelian category,
e.g. in the category of discrete $\ggg$-modules where $\ggg$ is a profinite group.
By a morphism $\varphi$ of exact sequences
we mean a commutative diagram
$$
\xymatrix{
0\ar[r]   &A\updot\ar[r]\ar[d]^{\varphi_A} &B\updot\ar[r]\ar[d]^{\varphi_B} &C\updot\ar[r]\ar[d]^{\varphi_C} &0\\
0\ar[r]   &A\upprimedot\ar[r]               &B\upprimedot\ar[r]             &C\upprimedot\ar[r]         &0.
}
$$
We say that $\varphi$ is a quasi-isomorphism of exact sequences if
$\varphi_A$, $\varphi_B$ and $\varphi_C$ are quasi-isomorphisms.
\end{subsec}

\begin{lemma}[well known]\label{lem:cone-hom}
 Consider a short exact sequence of complexes as in \eqref{eq:complexes}.
Define a morphism of complexes
$$
\lambda\colon\langle A\updot\to B\updot] \to C\updot\quad\text{by}\quad\lambda^i(a^{i+1},b^i)=\nu^i(b^i).
$$
Then $\lambda$ is indeed  a morphism of complexes and is a quasi-isomorphism, functorial in the  exact sequence  \eqref{eq:complexes}.
\end{lemma}
\begin{proof}
 See \cite[Ch. III, \S 3, Proof of Prop. 5]{GM}  or \cite[1.5.8]{Weibel} for a proof that $\lambda$ is a morphism of complexes and a quasi-isomorphism.
The functoriality is obvious.
\end{proof}

\begin{lemma}[well known]
\label{lem:quasi-isom}
Assume we have a commutative square  of complexes
$$
\xymatrix{
P\ar[d]_{\gamma_P}\ar[r]  &Q\ar[d]^{\gamma_Q}\\
P'\ar[r]                  &Q'
}
$$
where all the arrows are morphisms of complexes and
the two vertical arrows $\gamma_P$ and $\gamma_Q$ are quasi-isomorphisms.
Then the induced morphism of cones
$$
\langle P\to Q]\labelto{\gamma_*} \langle P'\to Q']
$$
is a quasi-isomorphism.
\end{lemma}

\begin{proof}
The lemma follows easily from the five-lemma.
\end{proof}

\begin{construction}\label{con:H0-B}
Let
\begin{equation}
0\to [0\to A^1\rangle \labelto{\mu} [B^0\to B^1\rangle
                         \labelto{\nu} [C^0\to C^1\rangle\to 0\tag{S}
\end{equation}
be a short exact sequence of complexes  in an abelian category.
We write
$$
A\updot=[0\to A^1\rangle,\quad B\updot=[B^0\labelto{d_B} B^1\rangle,\quad
                                          C\updot=[C^0\labelto{d_C} C^1\rangle,
$$
so we have an exact sequence of complexes
$$
0\to A\updot\labelto{\mu} B\updot\labelto{\nu} C\updot\to 0.
$$
Assume that the  following condition is satisfied:
\begin{equation}
\sH^0(B\updot)=0.\tag{$\sH^0$}
\end{equation}
Then we have a canonical quasi-isomorphism $B\updot\to\sH^1(B\updot)[-1]$.
By Lemma \ref{lem:quasi-isom} the induced morphism of cones
$$
\langle A\updot\to B\updot]\to\langle A\updot\to\sH^1(B\updot)[-1]]
$$
is a quasi-isomorphism.
Since $A\updot=A^1[-1]$, we obtain a quasi-isomorphism
$$
\varepsilon\colon \langle A\updot\to B\updot]\to [A^1\labelto{\sigma}\sH^1(B\updot)\rangle,
$$
where $\sigma(a^1)=-\mu^1(a^1) + \im\,d_B\in \sH^1(B\updot)$ \  for $a^1\in A^1$.

We write formulae for $\varepsilon$:
$$
\varepsilon^0(a^1,b^0)=a^1,\quad \varepsilon^1(b^1)= b^1 + \im d_B.
$$
\end{construction}

\begin{corollary}\label{cor:H0-B}
(a) Let $(S)$ be an exact sequence of complexes as in Construction \ref{con:H0-B},
satisfying  condition $(\sH^0)$.
Then there is a canonical, functorial in $(S)$ isomorphism in the derived category
$$
[A^1\labelto{\sigma}\sH^1(B\updot)\rangle \labelto{\sim} C\updot
$$
given by the diagram
$$
\xymatrix{
{[A^1\labelto{\sigma}\sH^1(B\updot)\rangle}   &\langle A\updot\to B\updot]\ar[l]_-{\varepsilon}\ar[r]^-{\lambda}  &C\updot
}
$$
where $\varepsilon$ is the quasi-isomorphism of Construction \ref{con:H0-B}
and $\lambda$ is the quasi-isomorphism of Lemma \ref{lem:cone-hom}.

(b) Let $\varphi\colon (S)\to (S')$ be a quasi-isomorphism of exact sequences as in
Construction \ref{con:H0-B},
and assume that both $(S)$ and $(S')$ satisfy the condition $(\sH^0)$.
Then in the commutative diagram
$$
\xymatrix{
{\left[A^1\labelto{\sigma}\sH^1(B\updot)\right\rangle}\ar[d]
                             &\left\langle A\updot\to B\updot\right]\ar[d]\ar[l]_-{\varepsilon}\ar[r]^-{\lambda}  &C\updot\ar[d]\\
{\left[A^{\prime 1}\labelto{\sigma'}\sH^1(B\upprimedot)\right\rangle}
                                  &\left\langle A\upprimedot\to B\upprimedot\right]\ar[l]_-{\varepsilon'}\ar[r]^-{\lambda'}  &C\upprimedot
}
$$
all the vertical arrows are quasi-isomorphisms.
\end{corollary}

\begin{proof}
The assertion (a) follows from Lemma \ref{lem:cone-hom} and Construction \ref{con:H0-B},
and the assertion (b) follows from Lemma \ref{lem:quasi-isom}.
\end{proof}

Now we apply Corollary \ref{cor:H0-B} to the exact sequence of Lemma \ref{lem:upicg-decomp}.

\begin{theorem}\label{thm:UG(Xbar)=0}
  Let $X$ and $G$ be as in \S\ref{subsec:G}.
Assume that the map $\Pic_G(\Xbar) \to\Pic(\Xbar)$ is surjective and that $U_G(\Xbar)=0$.
Then there is a canonical isomorphism  in the derived category
\begin{equation}\label{eq:ZalgOX-UPicG-UPicX}
    \left[ \Zalg^1(\Gbar, \OXbar) \labelto{\sigma} \Pic_G(\Xbar) \right\rangle \labelto{\sim} \UPic(\Xbar).
\end{equation}
This isomorphism is functorial in $X$ and $G$.
\end{theorem}
We specify the map $\sigma$ and the isomorphism in the derived category.
The map $\sigma$ takes a cocycle $c\in \Zalg^1(\Gbar, \OXbar)$ to the class of the trivial invertible sheaf $\sO_\Xbar$ on $\Xbar$
with the $G$-linearization given by $c$, see Lemma \ref{lem:G-linearizations}.
The isomorphism  \eqref{eq:ZalgOX-UPicG-UPicX} is given by the commutative diagram
\begin{equation}\label{eq:diagram-isomorphism}
 \xymatrix{
{\Zalg^1(\Gbar, \OXbar)}\ar[d]^\sigma  &{\Zalg^1(\Gbar, \OXbar)}\oplus\sK(\Xbar)^\times/\kbar^\times\ar[d]^\psi\ar[l]\ar[r]
                                                            &\sK(\Xbar)^\times/\kbar^\times\ar[d]^{\divisor}\\
\Pic_G(\Xbar)                          &\UPic_G(\Xbar)^1\ar[l]\ar[r]                  &\Div(\Xbar)
}
\end{equation}
where the arrow $\psi$ is given by
$$
\psi(c,[f])=(c\cdot d^0_\sK(f),\divisor(f))\in \UPic_G(\Xbar)^1\subset  \Zalg^1(\Gbar,\KXbar)\oplus\Div(\Xbar),
$$
and all the unlabeled arrows are the obvious ones.

\begin{proof}
  The isomorphism of the theorem is the isomorphism of Corollary \ref{cor:H0-B}(a)
applied to the exact sequence of Lemma~\ref{lem:upicg-decomp}.
  Functoriality in the case of a dominant morphism $f\colon X'\to X$ is evident.
In the case of a non-dominant $G$-morphism $f\colon X'\to X$
we use local modifications as in \S\ref{subsec:functoriality}.
\end{proof}

\section{Extended Picard complex of a homogeneous space}
\label{sec:hom-space}

Let $X$ be a homogeneous space under a connected $k$-group $G$.
In general $X$ may have no $k$-rational points,
hence not be of the form $G/H$.
The results of Section~\ref{sec:rel-upic_g-upic} give a nice description
of $\UPic(\Xbar)$ as long as $\Pic(\Gbar) = 0$.
Fortunately, the
latter assumption does not give any serious loss of generality
by virtue of  Lemma \ref{lem:epimorphism} below.

First let $G$ be a connected $k$-group acting on a geometrically integral  $k$-variety $X$.
The character group $\Chi(\Gbar)$ canonically embeds into $\Zalg^1(\Gbar, \OXbar)$.
We show that $\Zalg^1(\Gbar, \OXbar)=\Chi(\Gbar)$ using Rosenlicht's lemma.

\begin{lemma}\label{cor:Rosenlicht}
Let $X$ be a geometrically integral variety over an arbitrary field $k$,
and $G$ be a connected $k$-group.
Then every $u\in\sO(\Xbar\times_\kbar \Gbar)^\times$
can be uniquely written in the form
$u(x,g)=v(x)\chi(g)$, where $v\in\OXbar$ and $\chi\in\XX(\Gbar)$.\qed
\end{lemma}
\begin{proof}
The lemma follows easily from Rosenlicht's lemma, see e.g. \cite[Lemme 6.5]{Sansuc}.
\end{proof}

\begin{proposition}\label{cor:cocycle}
Let $G$ be a connected $k$-group acting on a geometrically integral
$k$-variety $X$.
Then $\Zalg^1(\Gbar, \OXbar)=\Chi(\Gbar)$.
\end{proposition}
\begin{proof}
Let $c\in\ZGOXbar$.
By Lemma \ref{cor:Rosenlicht} we may write
$$
c_g(x)=v(x)\chi(g) \text{ for some } v\in\OXbar,\ \chi\in\XX(\Gbar).
$$
Writing the cocycle condition \eqref{eq:cocycle} of \S \ref{subsec:calg-defs},
we obtain
$$
v(x)\chi(g_1 g_2)=v(x)\chi(g_1) v(xg_1)\chi(g_2),
$$
whence $v(xg_1)=1$.
Substituting $g_1=1$, we obtain $v(x)=1$.
Thus $c_g(x)=\chi(g)$ and $c=\chi$.
\end{proof}

Now let $X$ be a homogeneous space under a connected $k$-group $G$.
We compute $U_G(\Xbar)$ and (following Popov) $\Pic_G(\Xbar)$.

\begin{lemma}\label{lem:coh-upicg-homog}
  Let $X$ be a homogeneous space under a connected $k$-group $G$. Then
$U_G(\Xbar) = 0$.
\end{lemma}

\begin{proof}
  By definition, $U_G(\Xbar) = (\sO(\Xbar)^\times)^G/\kbar^\times$.
Since $G$ acts transitively, the only $G$-invariant functions are the constant
functions, and the lemma follows.
\end{proof}

\begin{subsec}\label{subsec:characters-homogeneous-spaces}
 Let $X$ be a homogeneous space under a connected $k$-group $G$,
and let $\xbar\in X(\kbar)$ be a $\kbar$-point.
Let $\Hbar\subset \Gbar$ be the stabilizer of $\xbar$ in $\Gbar$.
We define a homomorphism $\pi_\xbar\colon \Pic_G(\Xbar)\to\XX(\Hbar)$ as follows.
Let $(\sL,\beta)$ be an invertible sheaf  on $\Xbar$ with a $\Gbar$-linearization.
Consider the embedding $i\colon\Hbar \into \Xbar\times\Gbar,\ h\mapsto (\xbar,h)$.
Then $w\circ i=p_X\circ i\colon \Hbar\into \Xbar$,
hence
$$
(w\circ i)^*\sL = (p_X\circ i)^*\sL.
$$
We see that the $G$-linearization $\beta\colon w^*\sL\to p_X^*\sL$
gives an automorphism of the (trivial) invertible sheaf $(p_X\circ i)^*\sL$ on $\Hbar$,
and this automorphism is given by an invertible regular function $\chi$ on $\Hbar$.
The cocycle condition \eqref{eq:cocycle-beta}  of Definition \ref{def:lin-beta}
readily gives $\chi(h_1h_2)=\chi(h_1) \chi(h_2)$,
hence $\chi$ is a character of $\Hbar$.
We obtain a homomorphism $\pi_\xbar\colon \Pic_G(\Xbar)\to\XX(\Hbar)$, $\pi_\xbar([L,\beta])=\chi$.

Let $\xbar'\in X(\kbar)$ be another point and $\Hbar'$ its stabilizer.
We may write $\xbar'=\xbar g$ for some $g\in G(\kbar)$.
Then $\Hbar'=g^{-1} \Hbar g$, and we obtain an isomorphism
$\Hbar\to\Hbar'\colon h\mapsto g^{-1}hg$.
The induced isomorphism $g_*\colon \XX(\Hbar')\to\XX(\Hbar)$ does not depend on the choice of $g$,
and so we obtain a canonical identification of $\XX(\Hbar)$ with $\XX(\Hbar')$.
By Lemma \ref{lem:base-point} below, under this identification
the homomorphism $\pi_\xbar\colon \Pic_G(\Xbar)\to\XX(\Hbar)$ does not depend on $\xbar$.
\end{subsec}

\begin{lemma}\label{lem:base-point}
Let $G$, $X$, $\xbar$, $\xbar'$, and $g$ be as above, in particular, $\xbar'=\xbar g$.
Then for  $h\in\Hbar(\kbar)$ and $\beta\in \Pic_G(\Xbar)$ we have
$\pi_{\xbar'}(\beta)(g^{-1} h g)=\pi_\xbar(\beta)(h)$.
\end{lemma}

\begin{proof}
Omitted (it uses only the cocycle condition \eqref{eq:cocycle-beta} of Definition \ref{def:lin-beta}).
\end{proof}

\begin{proposition}[{\cite[Thm. 4]{Popov}}]\label{prop:picg-chi-h}
  Let $X$ be a homogeneous space under a connected $k$-group $G$,
and let $\xbar\in X(\kbar)$ be a $\kbar$-point.
Let $\Hbar \subset \Gbar$ be the stabilizer of $\xbar$.
Then the canonical homomorphism  of abelian groups
$\pi_\xbar\colon\Pic_G(\Xbar)\to\XX(\Hbar)$ of \S\ref{subsec:characters-homogeneous-spaces}
is an isomorphism.
\end{proposition}

\begin{proof}[Sketch of  proof]
Let $\chi\in\XX(\Hbar)$.
We define an embedding
$$
\theta\colon \Hbar\into\Gbar\times\Gmbar,\quad h\mapsto(h,\chi(h)^{-1}).
$$
Set $\Ybar=\theta(\Hbar)\backslash(\Gbar\times\Gmbar)$,
this is a quotient space of a linear algebraic group by an algebraic subgroup.
The group $\Gbar$ acts on $\Ybar$ on the right.
By Hilbert's Theorem 90 the principal $\Gmbar$-bundle  $\Ybar\to\Xbar$ admits a rational section,
and, since $\Gbar$ acts transitively on $\Xbar$, this bundle is locally trivial in the Zariski topology.
Using  transition functions for  the bundle $\Ybar\to\Xbar$, we define an invertible sheaf $\sL$ on $\Xbar$.
The action of $\Gbar$ on $\Ybar$ defines
a $\Gbar$-linearization $\beta$ of $\sL$.
We set $\pi'(\chi)=[\sL,\beta]\in\Pic_G(\Xbar)$.
We obtain a homomorphism $\pi'\colon \XX(\Hbar)\to \Pic_G(\Xbar)$, which is inverse to $\pi_\xbar$,
hence $\pi_\xbar$ is an isomorphism.
\end{proof}

\begin{remarks}\label{rem:res}
(1) Since the Galois group acts on $\Pic_G(\Xbar)$,
using Popov's isomorphism $\pi_\xbar$ of Proposition \ref{prop:picg-chi-h},
we can  endow $\XX(\Hbar)$ with a canonical structure of a Galois module.
This structure was earlier constructed by hand in \cite{Bor96}.

(2) We have a canonical homomorphism $\XX(\Gbar)\to\Pic_G(\Xbar)$
taking a character $\chi\in\XX(\Gbar)$ to the trivial invertible sheaf $\sL=\sO_\Xbar$
with the $\Gbar$-linearization $\beta_g(x)=\chi(g)\colon \sO_{\Xbar\times \Gbar}\to  \sO_{\Xbar\times \Gbar}$.
Clearly this homomorphism is a morphism of Galois modules.
On the other hand, under  Popov's identification $\Pic_G(\Xbar)=\XX(\Hbar)$
this homomorphism corresponds to the restriction map $\res\colon \XX(\Gbar)\to \XX(\Hbar)$
taking a character $\chi$ of $\Gbar$ to its restriction to $\Hbar$.
We see that the restriction map $\res$ is a morphism of Galois modules.

(3) Let $\xbar'\in X(\kbar)$ and $\Hbar'\subset \Gbar$ be as in \S\ref{subsec:characters-homogeneous-spaces}.
If we identify $\XX(\Hbar)$ and $\XX(\Hbar')$ as in \S\ref{subsec:characters-homogeneous-spaces},
then the map $\res\colon \XX(\Gbar)\to \XX(\Hbar)$
does not depend on the choice of the base point $\xbar$.
\end{remarks}

The following theorem is the main result of this paper.

\begin{theorem}\label{th:upic-hom-space}
  Let $X$ be a homogeneous space under a connected $k$-group $G$ with
   $\Pic(\Gbar) = 0$. Let $\Hbar$ be a geometric stabilizer.
  We have a canonical isomorphism in the derived category of Galois modules
  \begin{equation*}
    \UPic(\Xbar) \iso \left[\Chi(\Gbar) \labelto{\res} \Chi(\Hbar) \right\rangle
  \end{equation*}
 which is functorial in $G$ and $X$.
Here $\XX(\Hbar)$ has the Galois module structure given by the Popov's isomorphism,
and $\res$ is the restriction map.
\end{theorem}

Note that the complex of Galois modules $[\Chi(\Gbar) \labelto{\res} \Chi(\Hbar)\rangle$
and the isomorphism of Theorem \ref{th:upic-hom-space}
do not depend on the choice of the base point $\xbar$
(up to the canonical identification of \S\ref{subsec:characters-homogeneous-spaces}).
Indeed, $\Chi(\Hbar)$ does not depend on $\xbar$, see  \S\ref{subsec:characters-homogeneous-spaces}.
By Remark \ref{rem:res}(3) neither does the homomorphism $\res\colon \XX(\Gbar)\to \XX(\Hbar)$.
The Galois action on $\Chi(\Hbar)$ is given by Popov's isomorphism, hence does not depend on $\xbar$ either.
One can easily see from the proof of Theorem \ref{th:upic-hom-space} that the isomorphism of this theorem
does not depend on $\xbar$ (we use Lemma \ref{lem:base-point}).

\begin{proof}[Proof of Theorem \ref{th:upic-hom-space}]
By Lemma \ref{lem:coh-upicg-homog} $U_G(\Xbar)=0$.
By Theorem \ref{thm:UG(Xbar)=0}
there is a canonical isomorphism in the derived category of Galois modules
$$
 \left[ \Zalg^1(\Gbar, \OXbar)\to \Pic_G(\Xbar) \right\rangle \labelto{\sim} \UPic(\Xbar).
$$
By Proposition \ref{cor:cocycle} $\Zalg^1(\Gbar, \OXbar)=\XX(\Gbar)$.
By Proposition \ref{prop:picg-chi-h} $\Pic_G(\Xbar)\cong\XX(\Hbar)$.
The map $\sigma$ of Theorem \ref{thm:UG(Xbar)=0}
corresponds to the homomorphism $\res\colon[\Chi(\Gbar) \to \Chi(\Hbar) \rangle$,
and the theorem follows.
\end{proof}

\section{Extended Picard complex of a $k$-group}
\label{sec:group}

\begin{notation}\label{notation:G-subquotients}
Let $G$ be a connected $k$-group. Then:
\par $G\uu$ is the unipotent radical of $G$;
\par $G\red=G/G\uu$, it is a reductive group;
\par $G\sss$ is the derived group of $G\red$, it is semisimple.
\par $G\tor=G\red/G\sss$, it is a torus;
\par $G\scon$ is the universal covering of $G\sss$, it is simply connected.
\end{notation}

\begin{remark}\label{rem:Pic(Gbar)=0}
It is easy to see that $\Pic(\Gbar)=\Pic(\Gbar\sss)$.
We have $\Pic(\Gbar\sss)=\Chi(\ker[\Gbar\scon\to\Gbar\sss])$,
see \cite[Theorem 3]{Popov} or \cite[Corollary 4.6]{Fossum-Iversen}.
It follows that $\Pic(\Gbar)=0$ if and only if $G\sss$ is simply connected.
\end{remark}

\begin{subsec}
We have a canonical homomorphism
$$
\rho\colon G\scon\onto G\sss\into G\red
$$
(Deligne's homomorphism), which in general is neither injective nor surjective.
Let $T\subset G\red$ be a maximal torus.
Let $T\scon=\rho^{-1}(T)$ be the corresponding maximal torus of $T\scon$.
By abuse of notation we denote the restriction of $\rho$ to $T\scon$
again by $\rho\colon T\scon\to T$. We have a pullback morphism
$$
\rho^*\colon \XX(\Tbar)\to\XX(\Tbar\scon).
$$
\end{subsec}

The following theorem is the main result of our paper \cite{BvH07}.

\begin{theorem}[{\cite[Thm. 1]{BvH07}}]\label{thm:old-main-result}
Let $G$ be a connected  $k$-group.
Let  $T\subset G\red$ be a maximal torus.
Then there is a canonical isomorphism in the derived category
\begin{equation*}
    \UPic(\Gbar) \isoto \left[ \Chi(\Tbar) \labelto{\rho^*} \Chi(\Tbar\scon) \right\rangle,
  \end{equation*}
which is functorial in the pair $(G,T)$.
\end{theorem}

At the core of the proof in \cite{BvH07} were concrete but rather
tedious calculations.
Our present results allow us to give a new, more conceptual
proof of  this theorem.

\begin{definition}
An $m$-extension of a connected $k$-group $G$ is a central extension
$$
1\to M\to G'\to G\to 1
$$
where $\Pic(\Gbar')=0$ and $M$ is a $k$-group of multiplicative type.
\end{definition}

\begin{lemma}[well known]\label{lem:epimorphism}
Any connected $k$-group $G$ admits an $m$-extension.
\end{lemma}

\begin{proof}
We give a simple proof.
We use Notation \ref{notation:G-subquotients}.
Let $Z(G\red)^\circ$ denote the identity component of the center $Z(G\red)$ of $G\red$;
it is a $k$-torus.
Set $G''=G\scon\times Z(G\red)^\circ$.
We have an epimorphism
$$
\varsigma\colon G''\onto G\red\colon (g,z)\mapsto\rho(g)z \text{\quad for\quad } g\in G\scon, z\in  Z(G\red)^\circ.
$$
Since $k$ is a field of characteristic 0, there exists a splitting  of the extension
$$
 1\to G\uu\to G \to G\red\to 1
$$
(see \cite[Theorem 7.1]{Mostow}).
Such a  splitting defines an action of $G\red$  on $G\uu$,
and we have an isomorphism  $G\isoto G\uu\rtimes G\red$ (Levi decomposition).
The group $G''$ acts on $G\uu$ via $G\red$.
Set $G'=G\uu\rtimes G''$.
We have an epimorphism
$$
{\rm id}\times\varsigma\colon G'\onto G.
$$
Since $(G')\sss=(G'')\sss=G\scon$, we see that $(G')\sss$ is simply connected, hence $\Pic(\Gbar')=0$.
Clearly $M:=\ker[G'\to G]$ is a central, finite subgroup of multiplicative type in $G'$.
\end{proof}

Note that a much stronger assertion than Lemma \ref{lem:epimorphism} is known,
see Lemma \ref{lem:CT} below.
According to Colliot-Th\'el\`ene  \cite[Prop. 2.2]{CT06},
a quasi-trivial $k$-group is an extension of a quasi-trivial $k$-torus
by a simply connected $k$-group.
Recall  that a simply connected $k$-group is an extension
of a simply connected semisimple $k$-group by a (connected) unipotent $k$-group.

\begin{lemma}[Colliot-Th\'el\`ene]\label{lem:CT}
Let $G$ be a connected $k$-group.
Then there exists a central extension
$$
   1\to T\to G'\to G\to 1,
$$
where $G'$ is a quasi-trivial $k$-group and $T$ is a $k$-torus.
\end{lemma}

\begin{proof}
See \cite[Prop.-D\'ef. 3.1]{CT06}.
\end{proof}

\begin{lemma}\label{lem:for-functoriality}
Let $G_1\to G_2$ be a homomorphism of connected $k$-groups,
and let $G'_i\to G_i\ (i=1,2)$ be $m$-extensions.
Then there exists a commutative diagram
$$
\xymatrix{
G'_1\ar[d]   &G'_3\ar[l]\ar[d]\ar[r]   &G'_2\ar[d]  \\
G_1    &G_1\ar[l]_{{\rm id}}\ar[r]     &G_2
}
$$
in which $G'_3\to G_1$ is an $m$-extension.
\end{lemma}

\begin{proof}
Similar to \cite[Proof of Lemma 2.4.4]{Kot84}.
\end{proof}

The following theorem gives an alternative description of $\UPic(\Gbar)$
for a connected $k$-group $G$.

\begin{theorem}\label{thm:alternative-old-main-result}
Let $G$ be a connected $k$-group
 and let $X$ be a $k$-torsor  (a principal homogeneous space) under $G$.
Let
$$
1\to M\to G'\to G\to 1
$$
be an $m$-extension.
Then there is a canonical isomorphism,  functorial in the triple  $(G,G',X)$,
$$
\UPic(\Xbar)\isoto \left[\XX(\Gbar')\labelto{{\rm res}}\XX(\Mbar)\right\rangle
$$
in the derived category of discrete Galois modules.
Here ${\rm res}\colon\XX(\Gbar')\to\XX(\Mbar)$  is the restriction homomorphism.
\end{theorem}

\begin{proof}
The group $G'$ acts on $X$ via $G$, and $X$ is a homogeneous space of $G'$ with geometric stabilizer $\Mbar$.
Since  $\Pic(\Gbar')=0$, the theorem follows from Theorem \ref{th:upic-hom-space}.
\end{proof}

\begin{subsec}\label{subsec:second-proof}\emph{New proof of Theorem \ref{thm:old-main-result}.}
By \cite[Lemma 4.1]{BvH07} the canonical epimorphism $r\colon G\to G\red$
induces a canonical isomorphism in the derived category
$r^*\colon \UPic(\Gbar\red)\isoto\UPic(\Gbar)$, and therefore we may assume that $G$ is reductive.
Choose an $m$-extension
$$
1\to M\to G'\to G\to 1,
$$
then $G'$ is reductive as well.
By Theorem \ref{thm:alternative-old-main-result}
there is a canonical, functorial in $G$ and $G'$ isomorphism
$$
\UPic(\Gbar)\isoto [\XX(\Gbar')\to\XX(\Mbar)\rangle.
$$

Let $T\subset G$ be a maximal torus.
Let $T'$ denote the preimage of $T$ in $G'$,
then $T'$ is a maximal torus in the reductive group $G'$.
Consider the commutative diagram
\begin{equation}\label{eq:diagram-quasi-isomorphism-group}
\xymatrix{
{\XX(\Gbar')}\ar[d]\ar[r]   &\XX(\Mbar)\ar[d]  \\
\XX(\Tbar')\ar[r]           &\XX(\Mbar)\oplus\XX(\Tbar\scon)  \\
\XX(\Tbar)\ar[u]\ar[r]      &\XX(\Tbar\scon)\ar[u]
}
\end{equation}
with obvious arrows.
It is easy to check that this diagram gives quasi-isomorphisms
\begin{align*}
[\XX(\Gbar')\to   \XX(\Mbar)\rangle  &\labelto{} [\XX(\Tbar') \to \XX(\Mbar)\oplus\XX(\Tbar\scon)\rangle \\
[\XX(\Tbar)\to     \XX(\Tbar\scon)\rangle &\labelto{} [\XX(\Tbar') \to \XX(\Mbar)\oplus\XX(\Tbar\scon)\rangle,
\end{align*}
hence we obtain an isomorphism
$$
[\XX(\Gbar')\to   \XX(\Mbar)\rangle  \labelto{\sim}
[\XX(\Tbar)\to     \XX(\Tbar\scon)\rangle
$$
in the derived category.
Thus we obtain an isomorphism
$$
\UPic(\Gbar)\isoto
[\XX(\Tbar)\to     \XX(\Tbar\scon)\rangle
$$
in the derived category.
Using Lemma \ref{lem:for-functoriality}, one can easily see that this isomorphism
does not depend on the choice of the $m$-extension $G'$ of $G$  and is functorial in $(G,T)$.
\qed
\end{subsec}

\section{Picard and Brauer groups}
\begin{theorem}\label{cor:Pic-hom-space}
Let $X$ be a homogeneous space under a connected $k$-group $G$ with
  $\Pic(\Gbar) = 0$. Let $\Hbar$ be
the  stabilizer of a geometric point  $\xb \in X(\kbar)$
(we do not assume that $\Hbar$ is connected).
Then there is a canonical injection
$$
\Pic(X)\into H^1(k,[\XX(\Gbar)\to\XX(\Hbar)\rangle ),
$$
which is an isomorphism if $X(k)\neq\emptyset$ or $\Br(k)=0$.
\end{theorem}

\begin{proof}
This follows immediately from Theorem~\ref{th:upic-hom-space} and
  the exact sequence
\begin{equation}\label{eq:Pic-Br}
0\to\Pic(X)\to H^1(k, \UPic(\Xbar))\to \Br(k)\to \Br_1(X)\to H^2(k,\UPic(\Xbar))\to H^3(k,\Gm)
\end{equation}
established in \cite[Prop.~2.19]{BvH07} for an
arbitrary smooth geometrically integral variety.
Note that the homomorphisms $H^1(k, \UPic(\Xbar))\to \Br(k)$ and $H^2(k,\UPic(\Xbar))\to H^3(k,\Gm)$
in this exact sequence are zero if $X(k)\neq\emptyset$, see  \cite[Prop.~2.19]{BvH07}.
\end{proof}

  We also prove a conjecture
  \cite[Conj.~3.2]{Borovoi:hasse-homogeneous} of the first-named author
  concerning the subquotient $\Bra(X) = \ker [ \Br(X) \to
  \Br(\Xbar)] / \im [\Br(k) \to \Br(X)]$ of the Brauer group of a
  homogeneous space $X$.

\begin{theorem}\label{cor:br-hom-space}
Let $X$, $G$, and $\Hbar$ be as in Theorem \ref{cor:Pic-hom-space}.
Then there is  a canonical injection
  \begin{equation*}
  \Bra(X) \into H^2(k, [\Chi(\Gbar) \to \Chi(\Hbar) \rangle ),
  \end{equation*}
  which is an isomorphism if $X(k) \neq \emptyset$ or
  $H^3(k,\Gm) = 0$ (e.g., when $k$ is a number field or a
  $\mathfrak{p}$-adic field).
\end{theorem}
\begin{proof}
  This follows immediately from Theorem~\ref{th:upic-hom-space} and  exact sequence \eqref{eq:Pic-Br}.
\end{proof}

\begin{subsec}\label{subsec-torsor-Pic}
 We consider the special case when $X$ is a \emph{principal} homogeneous space of a connected $k$-group $G$
(we do not assume that $\Pic(\Gbar)=0$).
Let
$$
1\to M\to G'\to G\to 1
$$
be an $m$-extension.
Then, assuming that either $X$ has a $k$-point or $\Br(k)=0$,
we obtain from Theorem \ref{cor:Pic-hom-space} that
\begin{equation}\label{eq:Pic-torsor}
\Pic(X)\cong H^1(k, [\XX(\Gbar')\to\XX(\Mbar)\rangle).
\end{equation}
Let $T\subset G\red$ and $T\scon\subset G\scon$ be as in Notation \ref{notation:G-subquotients}
Using the isomorphism in the derived category
given by  diagram \eqref{eq:diagram-quasi-isomorphism-group} in \S\ref{subsec:second-proof},
we obtain from \eqref{eq:Pic-torsor} that
\begin{equation}\label{eq:Pic-torsor-tori}
 \Pic(X)\cong H^1(k, [\XX(\Tbar)\to\XX(\Tbar\scon)\rangle).
\end{equation}

Similarly, assuming that either $X$ has a $k$-point or $H^3(k,\Gm)=0$ (and not assuming that $\Br(k)=0$),
we obtain from Theorem \ref{cor:br-hom-space} that
\begin{equation}\label{eq:Br-torsor}
\Bra(X)\cong H^2(k, [\XX(\Gbar')\to\XX(\Mbar)\rangle)
\end{equation}
and
\begin{equation}\label{eq:Br-torsor-tori}
 \Bra(X)\cong H^2(k, [\XX(\Tbar)\to\XX(\Tbar\scon)\rangle).
\end{equation}

The  formulae \eqref{eq:Pic-torsor-tori} and \eqref{eq:Br-torsor-tori}  are  versions of our previous results \cite[Cor.~5 and Cor.~7]{BvH07}
and results of Kottwitz \cite[2.4]{Kot84}.
Note that when $G$ is a $k$-torus or a semisimple $k$-group,
formulae for $\Pic(G)$ and $\Bra(G)$ were earlier given by Sansuc \cite[Lemme 6.9]{Sansuc}.
\end{subsec}

\section{Comparison with topological invariants}
\begin{subsec}
  For a $k$-group $G$, the main result of~\cite{BvH07} and the
  comparison between the algebraic and the topological fundamental
  group of complex linear algebraic groups in \cite{Borovoi:Memoir},
  imply that the derived dual object
  \begin{equation*}
    \UPic(\Gbar)^D := R \Hom_\Zz(\UPic(\Gbar), \Zz)
  \end{equation*}
  is represented by a finitely generated Galois module concentrated in degree
  $0$ which, as an abelian group, is isomorphic to the fundamental
  group of the complex analytic space $G(\Cc)$  for any embedding
$\kbar \into \Cc$ of $\kbar$ into the complex numbers.
  For later use, let us make the further observation that in fact
  \begin{align*}
    U(\Gbar)^D =& \Chi(\Gbar)^D  = \pi_1(G(\Cc)) / \pi_1(G(\Cc))\tors, \\
    (\Pic(\Gbar)[-1])^D = &\Ext^1_\Zz(\Pic(\Gbar),\Zz)=\Hom(\Pic(\Gbar), \Qq/\Zz)  = \pi_1(G(\Cc))\tors,
  \end{align*}
see \eqref{eq:RiHom} below,
  where $M\tors$ denotes the torsion subgroup of a finitely generated
  abelian group $M$.

 For a homogeneous space $X$ under a $k$-group $G$, the derived dual
 object in general is not representable by a single group
 concentrated in degree $0$, as we see from the following lemma.
\end{subsec}

\begin{lemma}\label{lem:upicX-dual-les}
  Let $X$ be a homogeneous space with  geometric stabilizer  $\Hbar$
under a connected $k$-group $G$ with  $\Pic(\Gbar) = 0$.
  We have a long exact sequence of finitely generated Galois modules
  \begin{multline}\label{eq:upicX-dual-les}
    0 \to \sH^{-1}(\UPic(\Xbar)^D) \to \Hom_\Zz(\Chi(\Hbar), \Zz) \to
    \Hom_\Zz(\Chi(\Gbar), \Zz) \to \\
\sH^0(\UPic(\Xbar)^D) \to
    \Hom_\Zz(\Chi(\Hbar)\tors, \Qq/\Zz) \to 0
  \end{multline}
  and $\sH^{i}(\UPic(\Xbar)^D)=0$ for $i\neq 0,-1$.
\end{lemma}
\begin{proof}
  It follows from Theorem~\ref{th:upic-hom-space} that we have an exact triangle
  \begin{equation*}
     \Chi(\Hbar) [-1] \to \UPic(\Xbar) \to \Chi(\Gbar)\to\Chi(\Hbar)
  \end{equation*}
  in the derived category of Galois modules.
Applying $R \Hom_\Zz (\hvar,\Zz)$,
we get the dual triangle
$$
 R\Hom_\Zz(\Chi(\Hbar),\Zz) \to  R\Hom_\Zz(\Chi(\Gbar),\Zz) \to \UPic(\Xbar)^D\to R\Hom_\Zz(\Chi(\Hbar),\Zz)[1].
$$
Taking cohomology yields the exact sequence
\begin{multline*}
 0\to \sH^{-1} (\UPic(\Xbar)^D)  \to R^0\Hom_\Zz(\Chi(\Hbar),\Zz) \to   R^0\Hom_\Zz(\Chi(\Gbar),\Zz)\to \\
\sH^{0} (\UPic(\Xbar)^D)  \to R^1\Hom_\Zz(\Chi(\Hbar),\Zz) \to   R^1\Hom_\Zz(\Chi(\Gbar),\Zz),
\end{multline*}
which coincides with  exact sequence \eqref{eq:upicX-dual-les}, since
  \begin{equation}\label{eq:RiHom}
     R^i \Hom_\Zz(M, \Zz) = \begin{cases}
          \Hom_\Zz(M, \Zz) & \text{if $i = 0$,} \\
          \Ext^1_\Zz(M, \Zz) & \text{if $i = 1$,} \\
          0 & \text{otherwise}
        \end{cases}
  \end{equation}
  and $\Ext^1_\Zz(M, \Zz) = \Hom_\Zz(M\tors, \Qq/\Zz)$
  for any finitely generated abelian group $M$.
\end{proof}

\begin{subsec}\label{subsec:H1}
Let $\Hbar$ be a $\kbar$-group, not necessarily connected.
We denote by $\Hbar\mult$ the largest quotient group of multiplicative type of $\Hbar$,
then $\Chi(\Hbar\mult)=\Chi(\Hbar)$.
We set $\Hbar_1=\ker[\Hbar\to\Hbar\mult]$, then $\Hbar_1=\bigcap_{\chi\in\Chi(\Hbar)} \ker\,\chi$.
We consider the following condition on $\Hbar$:
\begin{equation}\label{eq:H1}
\Hbar_1 \text{ is connected and } \XX(\Hbar_1)=0.
\end{equation}
Note that this condition is satisfied if $\Hbar$ is connected.
\end{subsec}

\begin{lemma}\label{lem:H1-pi1}
Let $H$ be a $\Cc$-group satisfying the condition \eqref{eq:H1}.
Then there is a canonical, functorial in $H$ epimorphism
$$
\pi_1(H^\circ (\Cc))\to \Hom(\XX(H),\Zz)
$$
inducing an isomorphism
$$
\pi_1(H^\circ (\Cc))/\pi_1(H^\circ (\Cc))\tors \isoto \Hom(\XX(H),\Zz),
$$
where we denote by $H^\circ$ the identity component of $H$.
\end{lemma}

\begin{proof}
We have
$$
\Hom(\XX(H),\Zz)=\Hom(\XX(H\mult),\Zz)=\Hom(\XX((H\mult)^\circ ),\Zz)=\pi_1((H\mult)^\circ (\Cc)).
$$
Since $H$ satisfies the condition \eqref{eq:H1},
we have an exact sequence of connected $\Cc$-groups
$$
1\to H_1\to H^\circ \to (H\mult)^\circ \to 1
$$
and a homotopy exact sequence
$$
\pi_1(H_1(\Cc))\to\pi_1(H^\circ (\Cc))\to\pi_1((H\mult)^\circ (\Cc))\to\pi_0(H_1(\Cc))=1.
$$
From this exact sequence we obtain  an epimorphism
$$
\pi_1(H^\circ (\Cc))\to\pi_1((H\mult)^\circ (\Cc))=\Hom(\XX(H),\Zz).
$$
Since  $\pi_1(H_1(\Cc))$ is a finite group and $\Hom(\XX(H),\Zz)$ is torsion free, we obtain an isomorphism
$$
\pi_1(H^\circ (\Cc))/\pi_1(H^\circ (\Cc))\tors\isoto \Hom(\XX(H),\Zz).
$$
\end{proof}

\begin{theorem}\label{th:upicd-p1p2}
  Let $X$ be a homogeneous space
under a connected $k$-group $G$
with connected geometric stabilizers.
 Let us fix an embedding $\kbar \into \Cc$ and an isomorphism $\pi_1(\Cc^\times)\to \Zz$.
  \begin{theoremlist}
    \item
      We have an isomorphism of  groups
      \begin{equation*}
         \sH^0(\UPic(\Xbar)^D)  \iso \pi_1(X(\Cc)).
      \end{equation*}
     \item
      We have an isomorphism of abelian groups
       \begin{equation*}
         \sH^{-1}(\UPic(\Xbar)^D)  \iso \pi_2(X(\Cc))/\pi_2(X(\Cc))\tors.
       \end{equation*}
  \end{theoremlist}
\end{theorem}

\begin{proof}
By Lemma \ref{lem:CT} we can represent $X$ as a homogeneous space of
a connected $k$-group $G'$ with $\Pic(\Gbar')=0$ with connected geometric stabilizers,
see also \cite[Lemma 5.2]{Bor96}.
Therefore we may and shall assume that $\Pic(\Gbar)=0$.
 After base change to $\Cc$ we may assume $X = G/H$.
Since  $\pi_2(G(\Cc)) = 0$
(\'Elie Cartan, for the case of  compact Lie groups see \cite{Borel}),
the  long exact sequence of homotopy groups of a fibration gives us an
  exact sequence
\begin{equation*}
  0  \to \pi_2(X(\Cc)) \to \pi_1(H(\Cc)) \to \pi_1(G(\Cc))\to\pi_1(X(\Cc))\to 1.
\end{equation*}

\noindent\emph{Proof of (i).}
We have a commutative diagram with exact top row
$$
\xymatrix{
{\pi_1(H(\Cc))}\ar[d]\ar[r]      &\pi_1(G(\Cc))\ar[d]\ar[r]  &\pi_1(X(\Cc))\ar[r] &1\\
\Hom(\XX(H),\Zz)\ar[r]           &\Hom(\XX(G),\Zz)
}
$$
where by Lemma \ref{lem:H1-pi1} both vertical arrows are epimorphisms.
Since $\Pic(\Gbar)=0$, the group $\pi_1(G(\Cc))$ is torsion free,
and by Lemma \ref{lem:H1-pi1}  the right vertical arrow is an isomorphism.
Thus the diagram gives an exact sequence
$$
\Hom(\XX(H),\Zz)\to           \Hom(\XX(G),\Zz)\to\pi_1(X(\Cc))\to 1.
$$
From this exact sequence and Lemma \ref{lem:upicX-dual-les} we obtain an isomorphism
$$
\pi_1(X(\Cc))\cong\sH^0(\UPic(\Xbar)^D)
$$
because $\Hom(\XX(\Hbar)\tors,\Qq/\Zz)=0$ for a connected group $\Hbar$.
\medskip

\noindent\emph{Proof of (ii).} See Proposition \ref{prop:upicd-p1p2-H1} below.
\end{proof}

\begin{proposition}\label{prop:upicd-p1p2-H1}
  Let $X$ be a homogeneous space
under a connected $k$-group $G$ with $\Pic(\Gbar)=0$.
Let $\Hbar$ be a geometric stabilizer,
and assume that $\Hbar$ satisfies the condition   \eqref{eq:H1}.
 Let us fix an embedding $\kbar \into \Cc$ and an isomorphism $\pi_1(\Cc^\times)\to \Zz$.
  \begin{theoremlist}
    \item
 We have an isomorphism of abelian groups
       \begin{equation*}
         \sH^{-1}(\UPic(\Xbar)^D)  \iso \pi_2(X(\Cc))/\pi_2(X(\Cc))\tors.
       \end{equation*}
    \item
       If, moreover, $\Pic(\Hbar^\circ) = 0$, then we have
      \begin{equation*}
         \sH^{-1}(\UPic(\Xbar)^D)  \iso \pi_2(X(\Cc)).
       \end{equation*}
  \end{theoremlist}
\end{proposition}

\begin{proof}
(i) By Lemma \ref{lem:H1-pi1} we have
$$
\Hom(\XX(\Hbar), \Zz) = \pi_1(H^\circ (\Cc))/\pi_1(H^\circ (\Cc))\tors.
$$
Since $G$ is connected and $\Pic(\Gbar)=0$, we have
$$
\Hom(\XX(\Gbar), \Zz) = \pi_1(G(\Cc)).
$$
We obtain from
the fibration exact sequence
\begin{equation}\label{eq:homotopy-sequence1}
  0  \to \pi_2(X(\Cc)) \to \pi_1(H^\circ(\Cc)) \to \pi_1(G(\Cc))\to\pi_1(X(\Cc))\to\pi_0(H(\Cc))\to 1
\end{equation}
that
$$
\pi_2(X(\Cc))/\pi_2(X(\Cc))\tors=\ker[\Hom(\XX(\Hbar),\Zz) \to \Hom(\XX(\Gbar), \Zz)],
$$
hence $\sH^{-1}(\UPic(\Xbar)^D)  \iso \pi_2(X(\Cc))/\pi_2(X(\Cc))\tors$
by Lemma~\ref{lem:upicX-dual-les}.

(ii) If $\Pic(\Hbar^\circ ) = 0$, then $\pi_1(H^\circ (\Cc))$ is torsion free,
hence so is $\pi_2(X(\Cc))$ by the fibration exact sequence \eqref{eq:homotopy-sequence1}.
\end{proof}

Note that some condition on stabilizers in Theorem \ref{th:upicd-p1p2} must be imposed,
as shown by the following example:

\begin{example}\label{ex:counter-ex-H1}
Take $G={\rm SL}_{n,k}$, take $T$ to be the group of diagonal matrices in $G$,
and take $H=N$ to be  the normalizer of $T$ in $G$.
We have $H^\circ =T$.
Set $X=H\backslash G$.
Then
$$
\pi_2(X(\Cc))/\pi_2(X(\Cc))\tors=\pi_1(H^\circ(\Cc))/\pi_1(H^\circ(\Cc))\tors=\pi_1(T(\Cc))\cong\Zz^{n-1}.
$$
On the other hand, it follows from the next Lemma \ref{lem:N} that
$\Hom(\XX(\Hbar),\Zz)=0$, hence $\sH^{-1}(\UPic(\Xbar)^D)=0$.
\end{example}

\begin{lemma}[well-known]\label{lem:N}
Let $G={\rm SL}_{n,k}$, where $n\ge 2$,
let $T$ be the group of diagonal matrices in $G$,
and let $N$ to be  the normalizer of $T$ in $G$.
Then $\XX(\Nbar)\cong\Zz/2\Zz$.
 \end{lemma}

\begin{proof}
Let $t=\diag(z,z^{-1},1,\dots, 1)\in T(\kbar)$ where $z\in\kbar^\times$,
and let $s\in N(\kbar)$ be a representative of the transposition $(1,2)\in S_n=N/T$.
Then $sts^{-1}=t^{-1}$, hence $tst^{-1}s^{-1}=t^2=\diag(z^2,z^{-2},1,\dots, 1)$.
It follows easily that $T\subset N\der$, where $N\der$ denotes the derived group of $N$.
We see that $\XX(\Nbar)=\XX(\Nbar/\Tbar)=\XX(S_n)$.
Since $(S_n)\der=A_n$, we see that $\XX(S_n)=\XX(S_n/A_n)\cong \Zz/2\Zz$,
hence $\XX(\Nbar)\cong\Zz/2\Zz$.
\end{proof}

\section{The elementary obstruction}
\begin{subsec}
Recall that in \cite[Def.~2.10]{BvH07} the
\emph{elementary obstruction}
(to the existence of a $k$-point in a
 smooth geometrically integral $k$-variety $X$)
was defined as the class $e(X)\in\Ext^1(\UPic(\Xbar),\kbar^\times)$
associated to the extension
of complexes of Galois modules
$$
0\to\kbar^\times\to\left({\sK(\Xbar)}^\times\to \Div(\Xbar)\right)\to
       \left({\sK(\Xbar)}^\times/\kbar^\times\to \Div(\Xbar)\right)\to 0.
$$
It is a variant of the original elementary obstruction  ${\rm ob}(X)$ of
Colliot-Th\'el\`ene and Sansuc \cite[D\'ef.\ 2.2.1]{CT-Sansuc} which
lives in $\Ext^1({\sK(\Xbar)}^\times/\kbar^\times, \kbar^\times)$.
In fact, we have a canonical injection
$\Ext^1(\UPic(\Xbar),\kbar^\times) \to
\Ext^1({\sK(\Xbar)}^\times/\kbar^\times, \kbar^\times)$ which sends
$e(X)$ to ${\rm ob}(X)$
(see \cite[Lemma 2.12]{BvH07}).
\end{subsec}

\begin{subsec}\label{subsec:tor-and mult}
Recall that $\Hbar\mult$ denotes the largest quotient group of $\Hbar$
that is a group of multiplicative type.
We have $\Chi(\Hbar\mult)=\Chi(\Hbar)$.
The Galois action on $\Chi(\Hbar)$ defines a $k$-form $H\mm$ of $\Hbar\mult$.
Since we have a morphism of Galois modules $\Chi(\Gbar)\to\Chi(\Hbar)$,
we obtain a $k$-homomorphism  $i_*\colon H\mm\to G\tor$
(related to the embedding $i\colon \Hbar\into\Gbar$).

Set $G\ssu=\ker[G\to G\tor]$, it is an extension of a semisimple group by a unipotent group.
Set $\Hbar_1=\ker[\Hbar\to \Hbar\mult]$.
The embedding $\Hbar\into \Gbar$ gives a homomorphism $\Hbar\to\Gbar\tor$,
and $\Hbar_1$ is contained in the kernel of this homomorphism
(because the restriction of any character of $\Hbar$ to $\Hbar_1$ is trivial).
Thus $\Hbar_1\subset \Gbar\ssu$.
It is clear that $i_*\colon H\mm\to G\tor$ is an embedding if and only if
$\Hbar_1=\Hbar\cap\Gbar\ssu$.
\end{subsec}

\begin{proposition}\label{cor:comp-ext1-upic-gm}
    Let $X$ be a homogeneous space under a connected $k$-group $G$ with
   $\Pic(\Gbar) = 0$. Let $\Hbar$ be a geometric stabilizer.
  For every integer $i$  we have  a canonical isomorphism
  \begin{equation*}
    \Ext^i(\UPic(\Xbar), \kbar^\times) \iso
    H^i(k, \langle H\mm \to G\tor]),
  \end{equation*}
 which is functorial in $G$ and $X$.
\end{proposition}
\noindent

We need a lemma.

\begin{lemma}\label{lem:complexes}
Let $\ggg$ be a profinite group.
Let $\Ybul$ be a bounded complex of discrete $\ggg$-modules with finitely generated (over $\Zz$) cohomology.

(i) There exists a quasi-isomorphism $\psi\colon \Mbul \to \Ybul$,
where $\Mbul$ is a bounded complex  of finitely generated
(over $\Zz$) torsion free  $\ggg$-modules.

(ii) For any two such resolutions (quasi-isomorphisms)
 $\psi_1\colon \Mbul_1 \to \Ybul$ and  $\psi_2\colon \Mbul_2 \to \Ybul$ of $\Ybul$,
there exists a third such complex $\Mbul_3$ and quasi-isomorphisms   $\mu_j\colon \Mbul_3 \to \Mbul_j$ ($j=1,2$),
such that the following diagram commutes up to a homotopy:
$$
\xymatrix{
M_3\updot\ar[r]^{\mu_1}\ar[d]_{\mu_2}    &\Mbul_1\ar[d]^{\psi_1}\\
\Mbul_2\ar[r]^{\psi_2}                         &\Ybul.
}
$$
\end{lemma}

\begin{proof}
Assume that $Y^i=0$ for $i>n$.
We choose a finite set of generators $h_1,\dots,h_\kappa$ (over $\Zz$) of $\sH^n(\Ybul)$.
We lift each $h_j$ for $j=1,\dots,\kappa$ to $y_j\subset\ker[Y^n\to Y^{n+1}]$.
Let $\hh_j\subset \ggg$ denote the stabilizer of $y_j$ in $\ggg$,
it is an open subgroup (hence of finite index) in $\ggg$.
We have a canonical $\ggg$-morphism $\Zz[\ggg/\hh_j]\to \ker[Y^n\to Y^{n+1}]$
taking the image of the unit element $e\in\ggg$ in $\ggg/\hh_j$ to $y_j$.
Set $A^n=\bigoplus_j \Zz[\ggg/\hh_j]$,
then $A^n$ is a finitely generated (over $\Zz$) torsion free $\ggg$-module.
We have a morphism of $\ggg$-modules $A^n\to \ker[Y^n\to Y^{n+1}]$
such that the induced morphism $A^n\to \sH^n(\Ybul)$ is surjective.
We consider the complex $A^n[-n]$ (with one $\ggg$-module $A^n$ in degree $n$).
We have a morphism of complexes $\varphi\colon A^n[-n]\to \Ybul$.
We set $Y_{(1)}\updot=\langle A^n[-n]\to \Ybul]$ (the cone of $\varphi$).
It is easy to see that $\sH^n(Y_{(1)}\updot)=0$.
Then we apply this procedure to $Y_{(1)}\updot$ for  $n-1$
to obtain $Y_{(2)}\updot$  with $\sH^{n-1}(Y_{(2)}\updot)=0$, and so on.

Assume that $Y^i=0$ for $i\le n-m$ for some integer $m>0$.
Then $Y^i_{(m)}=0$ for $i<n-m$, and $Y^{n-m}_{(m)}$ is finitely generated and torsion free.
By construction $\sH^i(\Ybul_{(m)})=0$ for $i>n-m$.
Moreover, since $Y^{n-m-1}_{(m)}=0$, we have
$$
\sH^{n-m}(\Ybul_{(m)})=\ker[Y^{n-m}_{(m)}\to Y^{n-m+1}_{(m)}].
$$
Set $A'=\sH^{n-m}(\Ybul_{(m)})$, then $A'$ is finitely generated and torsion free,
because it is a subgroup of the finitely generated torsion free abelian group $Y^{n-m}_{(m)}$.
We have an {\em injective} morphism of $\ggg$-modules $A'\into Y^{n-m}_{(m)}$
and a morphism of complexes $\varphi'\colon A'[n-m]\to \Ybul_{(m)}$.
As before, set $\Ybul_{(m+1)}=\langle A'[n-m]\to \Ybul_{(m)}]$ (the cone of $\varphi'$).
One can easily see that the complex $Y_{(m+1)}\updot$ is acyclic.

One can check that $Y_{(m+1)}\updot$ is the cone of some morphism of complexes $\psi\colon \Mbul\to \Ybul$,
where $\Mbul$ is a bounded complex of finitely generated torsion free $\ggg$-modules.
Since the cone  $Y_{(m+1)}\updot$  of $\psi$ is acyclic, we see that $\psi$ is a quasi-isomorphism.

(ii) Let $N\updot:=[M_1\updot\oplus M_2\updot\labelTo{\psi_1-\psi_2} \Ybul\rangle$ (the fibre of $\psi_1-\psi_2$),
then we have morphisms
$$
\lambda_j\colon N\updot\to \Mbul_1\oplus\Mbul_2\to\Mbul_j,\ j=1,2.
$$
From the short exact sequence of complexes
$$0\to [\Mbul_2\to\Ybul\rangle\to N\updot\labelto{\lambda_1}\Mbul_1\to 0$$
where $[\Mbul_2\to\Ybul\rangle$ is acyclic because $\psi_2\colon \Mbul_2\to\Ybul$ is a quasi-isomorphism,
we see that $\lambda_1$ is a quasi-isomorphism, and similarly $\lambda_2$ is a quasi-isomorphism.
An easy calculation shows that the following diagram commutes up to a homotopy:
$$
\xymatrix{
N\updot\ar[r]^{\lambda_1}\ar[d]_{\lambda_2}    &\Mbul_1\ar[d]^{\psi_1}\\
\Mbul_2\ar[r]^{\psi_2}                         &\Ybul.
}
$$
Now we apply (i) to the complex $N\updot$ and obtain a quasi-isomorphism $\varkappa\colon\Mbul_3\to N\updot$, where
$\Mbul_3$ is a bounded complex  of finitely generated (over $\Zz$) torsion free  $\ggg$-modules.
We set $\mu_j=\lambda_j\circ\varkappa\colon \Mbul_3 \to M_j\updot$.
\end{proof}

\begin{proof}[Proof of Proposition \ref{cor:comp-ext1-upic-gm}]
Since the complex $\UPic(\Xbar)$ is bounded and, for  $X$ as in the proposition,  has finitely generated cohomology,
by Lemma \ref{lem:complexes}(i) there is a bounded resolution
$\psi\colon\Mbul\to \UPic(\Xbar)$ consisting of finitely generated torsion free Galois modules.
We have a canonical isomorphism
  \begin{equation*}
  \psi^*\colon  \Ext^i(\UPic(\Xbar), \kbar^\times) \isoto
    \Ext^i(\Mbul, \kbar^\times).
  \end{equation*}
It is well known (see for example~\cite[Lem.~1.5]{BvH07})
that there is a canonical isomorphism
\begin{equation*}
  \Ext^i(\Mbul,\kbar^\times)\isoto H^i(k,\Hombul_\Zz(\Mbul,\kbar^\times)),
\end{equation*}
hence we obtain an isomorphism
  \begin{equation*}
    \Ext^i(\UPic(\Xbar), \kbar^\times) \isoto
   H^i(k, R\sHom_\Zz(\UPic(\Xbar), \kbar^\times)),
 \end{equation*}
which does not depend on the choice of the resolution
$\psi\colon\Mbul\to \UPic(\Xbar)$ by Lemma \ref{lem:complexes}(ii).
It now follows from Theorem~\ref{th:upic-hom-space} that we have a canonical isomorphism
  \begin{equation*}
    \Ext^i(\UPic(\Xbar), \kbar^\times) \iso
   H^i(k, R\sHom_\Zz([\Chi(\Gbar) \to \Chi(\Hbar) \rangle, \kbar^\times)).
 \end{equation*}
Since $\kbar^\times$ is divisible,
\begin{align*}
 R\sHom_\Zz([\Chi(\Gbar) \to \Chi(\Hbar) \rangle, \kbar^\times) &=
\langle \Hom_\Zz(\Chi(\Hbar), \kbar^\times)\to \Hom_\Zz(\Chi(\Gbar), \kbar^\times)]\\
&=\langle H\mm(\kbar) \to G\tor(\kbar)].
\end{align*}
Thus  we obtain a canonical isomorphism
  \begin{equation*}
    \Ext^i(\UPic(\Xbar), \kbar^\times) \iso
    H^i(k, \langle H\mm \to G\tor]).
  \end{equation*}
\end{proof}

\begin{subsec}\label{subsec:cohomological-obstruction}
In \cite{Borovoi:hasse-homogeneous} the first-named author defined (by
means of explicit cocycles) an obstruction class
to the existence of a rational point on $X$.
We recall the definition.

Fix $\xbar\in X(\kbar)$.
Let $\Hbar\subset G_\kbar$ be the stabilizer of $\xbar$.
For each $\sigma\in\Gal(\kbar/k)$ choose $g_\sigma\in G(\kbar)$
such that ${}^\sigma\xbar=\xbar\cdot g_\sigma$.
We may assume that the map $\sigma\mapsto g_\sigma$
is continuous (locally constant) on $\Gal(\kbar/k)$.
We denote by $\hat g_\sigma\in G\tor(\kbar)$ the image of $g_\sigma$ in $G\tor(\kbar)$.
We obtain a continuous map $\hat g\colon\sigma\mapsto \hat g_\sigma$.

Let $\sigma,\tau\in \Gal(\kbar/k)$.
Set $u_{\sigma,\tau}=g_{\sigma\tau}(g_\sigma{}^\sigma g_\tau)^{-1}\in G(\kbar)$,
then it is easy to check that $u_{\sigma, \tau}\in \Hbar(\kbar)$.
We denote by $\hat u_{\sigma, \tau}\in H\mm (\kbar)$ the image of $ u_{\sigma, \tau}$ in $H\mm (\kbar)$ .
We obtain a continuous map $\hat u\colon (\sigma,\tau)\mapsto \hat u_{\sigma, \tau}$.
One can check that  $(\hat u, \hat g)\in Z^1(k, H\mm\to G\tor)$.
We denote by $\eta(G,X)\in H^1(H\mm\to G\tor)$ the hypercohomology class of the hypercocycle $(\hat u, \hat g)$,
see \cite{Borovoi:hasse-homogeneous} for details.
\end{subsec}

\begin{theorem}\label{thm:eta(G,X)}
 Let $X$ be a homogeneous space under a connected $k$-group $G$ with
 $\Pic(\Gbar) = 0$.
Let $\Hbar$ be a geometric stabilizer.
Then $e(X)\in\Ext^1(\UPic(\Xbar),\kbar^\times)$ coincides with
$-\eta(G,X) \in  H^1(k, \langle H\tor \to G\tor])$ under the
identification
$$
 \Ext^1(\UPic(\Xbar), \kbar^\times) \iso
    H^1(k, \langle H\mm \to G\tor])
$$
of Proposition~\ref{cor:comp-ext1-upic-gm}.
\end{theorem}

We shall prove Theorem \ref{thm:eta(G,X)} by \emph{d\'evissage} in three steps,
using an idea of \cite{Bor96}.

\begin{subsec}\label{subsec:proof-torus}
We prove Theorem \ref{thm:eta(G,X)} when $G$ is a torus.
The geometric stabilizer of a point $\xb\in X(\kbar)$ does not depend on $\xb$
and is defined over $k$; we denote the corresponding $k$-group by $H$.
We have $G\tor=G$, $H\mm=H$.
Set $T=G/H$, it is a $k$-torus, and $X$ is a torsor under $T$.
We have a canonical
morphism of complexes of abelian $k$-groups $\lambda\colon\langle H\to G]\to T$,
which is a quasi-isomorphism.

We have a commutative diagram
$$
\xymatrix{
{\Ext^1}(\UPic(\Xbar), \kbar^\times)\ar[r]^{\beta_{G,X}}\ar@{=}[d]
                          &   H^1(k,\langle H\to G])\ar[d]^{\lambda_*}\\
{\Ext^1}(\UPic(\Xbar), \kbar^\times)\ar[r]^{\beta_{T,X}}
                          & H^1(k,T)
}
$$
in which all the arrows are isomorphisms.
Since $\eta(G,X)$ is functorial in $(G,X)$, we see that
$\lambda_*(\eta(G,X))=\eta(T,X)$.
By \cite[Thm.~5.5]{BvH07} $\beta_{T,X}(e(X))=-\eta(T,X)$.
It follows that $\beta_{G,X}(e(X))=-\eta(G,X)$ (because $\lambda_*$ is an isomorphism).\qed
\end{subsec}

\begin{subsec}\label{subsec:proof-(*)}
Now we prove Theorem \ref{thm:eta(G,X)} assuming that the pair $(G,X)$ satisfies the following condition:
\begin{equation}\label{eq:*2}
 H\mm \text{ \ embeds into }G\tor.
\end{equation}

Set $Y=X/G\ssu$, see \cite[Lemma 3.1]{Bor96} for a proof that this quotient exists.
Then the torus $G\tor$ acts transitively on $Y$,
and it follows from our condition \eqref{eq:*2} that the stabilizer in $\Gbar\tor$
of any point $\bar{y}\in\Ybar$ is $\Hbar\mult$.
We have a morphism of pairs  $t\colon (G,X)\to(G\tor,Y)$, which gives rise to a commutative diagram
$$
\xymatrix{
{\Ext^1}(\UPic(\Xbar), \kbar^\times)\ar[r]^{\beta_{G,X}}\ar[d]^{t_*}
                          &   H^1(k,\langle H\mm\to G\tor])\ar[d]^{t_*}\\
{\Ext^1}(\UPic(\Ybar), \kbar^\times)\ar[r]^{\beta_{G\tor,Y}}
                          & H^1(k,\langle H\mm\to G\tor]) .
}
$$
Since $G\tor$ is a torus, by \ref{subsec:proof-torus} we have
$\beta_{G\tor,Y}(e(Y))=-\eta(G\tor,Y)$.
The right hand vertical map $t_*$ is the identity map.
It follows from the functoriality of $e(X)$ and $\eta(G,X)$
that $\beta_{G,X}(e(X))=-\eta(G,X)$.\qed
\end{subsec}

 In order to complete the proof we need a construction of \cite[Section 4]{Bor96},
see also \cite[Proof of Thm. 3.5]{BCTS07}.

\begin{proposition}\label{prop:main-trick-dominating}
Let $X$ be a homogeneous space of a connected $k$-group $G$ over a field $k$
of characteristic 0.
Then there exists a morphism of pairs $(\varphi, \pi)\colon (F,Z)\to (G,X)$,
where $Z$ is a homogeneous space of a $k$-group $F$, with the following properties:

(a) $F=G\times P$, and $\varphi\colon G\times P\to G$ is the projection,
where $P$ is a quasi-trivial $k$-torus;

(b) $\pi\colon Z\to X$ is a torsor under $P$, where $P$ acts on $Z$ through the injection $P\into G\times P=F$;

(c) the pair  $(F,Z)$ satisfies the condition \eqref{eq:*2} of \S\ref{subsec:proof-(*)}.
\end{proposition}

\begin{proof} (cf. \cite[Proof of Thm. 3.5]{BCTS07}.)
Let $\xb\in X(\kbar)$ be a $\kbar$-point with stabilizer $\Hbar$.
Let $\mu\colon \Hbar\to \Hbar\mult$ be the canonical surjection.
Let $\xi\in H^2(k,H\mm)$ be the image of
$\eta(G,X)\in H^1(k,\langle H\mm\to G\tor])$.
By \cite[Lemma 3.7]{BCTS07} there exists an embedding $j\colon H\mm\into P$
of $H\mm$ into a quasi-trivial torus $P$ such that $j_*(\xi)=0$.

Consider the $k$-group $F=G\times P$, and the embedding
$$
\Hbar\into \Fbar \quad \text{given by} \quad h\mapsto (h, j(\mu(h))).
$$
Set $\Zbar=\Hbar\backslash \Fbar$.
We have a right action $\abar\colon \Zbar\times\Fbar\to \Zbar$
and an $\Fbar$-equivariant map
$$
\pibar\colon \Zbar\to \Xbar,\quad \Hbar\cdot (g,p)\mapsto \Hbar\cdot g,
\text{\quad where }g\in \Gbar,\ p\in \overline{P}.
$$
Then $\Zbar$ is a homogeneous space of $\Fbar$ with respect to the action $\abar$,
and the map $\pibar\colon \Zbar\to \Xbar$ is a torsor under $\Pbar$.
The homomorphism $M\to F\tor$ is injective.

In \cite[4.7]{Bor96} it was proved that
$\text{Aut}_{\Fbar,\Xbar}(\Zbar)=P(\kbar)$.
By \cite[Lemma 4.8]{Bor96} the element $j_*(\xi)\in H^2(k, P)$
is the only obstruction to the existence
of a $k$-form $(Z,a,\pi)$ of the triple $(\Zbar, \abar,\pibar)$:
there exists such a $k$-form if and only if $j_*(\xi)=0$.
In our case by construction we have $j_*(\xi)=0$, hence
there exists a $k$-form $(Z,a,\pi)$ of  $(\Zbar, \abar,\pibar)$.
This completes the proof of the proposition.

Note that in \cite{Bor96} and \cite{BCTS07} we always assumed that
$\Hbar_1$ was connected and had no nontrivial characters,
but this assumption was not used in the cited constructions.
\end{proof}

\begin{subsec}
We prove Theorem \ref{thm:eta(G,X)} in the general case.
Consider a morphism of pairs
$$
(\varphi, \pi)\colon (F,Z)\to (G,X)
$$
as in Proposition \ref{prop:main-trick-dominating}.
Since $\Pic(\Gbar)=0$ and $F=G\times P$, we have $\Pic(\Fbar)=0$.
We have a commutative diagram
$$
\xymatrix{
{\Ext^1}(\UPic(\Zbar), \kbar^\times)\ar[r]^{\beta_{F,Z}}\ar[d]^{\pi_*}
                          &   H^1(k,\langle H\mm\to F\tor])\ar[d]^{\varphi_*}\\
{\Ext^1}(\UPic(\Xbar), \kbar^\times)\ar[r]^{\beta_{G,X}}
                          & H^1(k,\langle H\mm\to G\tor]) .
}
$$
Since the pair $(F,Z)$ satisfies the condition \eqref{eq:*2},
by \ref{subsec:proof-(*)} $\beta_{F,Z}(e(Z))=-\eta(F,Z)$.
It follows easily from the functoriality of $e(X)$ and $\eta(G,X)$
that $\beta_{G,X}(e(X))=-\eta(G,X)$.
This completes the proof of Theorem \ref{thm:eta(G,X)}.
\qed
\end{subsec}

Using Theorem \ref{thm:eta(G,X)} we can give new proofs for the following results of \cite{BCTS07}.

\begin{corollary}[{\cite[Theorem~3.5]{BCTS07}}]
\label{cor:over-p-adic}
Let $k$ be a $p$-adic field.
  Let $X$ be a homogeneous space under a connected (linear) $k$-group $G$
with connected geometric stabilizers.
Then $X(k) \neq \emptyset$ if and only if $e(X) = 0$.
\end{corollary}

\begin{proof}
By \cite[Lemma 5.2]{Bor96} we may  assume that $\Pic(\Gbar)=0$.
If $X(k)\neq\emptyset$, then clearly $e(X)=0$.
Conversely, if $e(X)=0$, then by Theorem~\ref{thm:eta(G,X)} $\eta(G,X)=0$,
hence by  \cite[Thm.~2.1]{Borovoi:hasse-homogeneous} $X(k)\neq\emptyset$.
\end{proof}

\begin{corollary}[{\cite[Theorem~3.10]{BCTS07}}]
Let $k$ be a number field.
  Let $X$ be a homogeneous space under a connected (linear) $k$-group $G$
with connected geometric stabilizers.
  Assume $X(k_v) \neq \emptyset$ for every real place $v$ of $k$.
Then $X(k) \neq \emptyset$ if and only if $e(X) = 0$.
\end{corollary}
\begin{proof}
Similar to the proof of Corollary \ref{cor:over-p-adic},
but using \cite[Cor.~2.3]{Borovoi:hasse-homogeneous}
instead of \cite[Thm.~2.1]{Borovoi:hasse-homogeneous}.
\end{proof}


\end{document}